\documentclass[12pt]{article}
\usepackage{amssymb,epsfig,amsfonts,epic,graphics,pictex, 
} 
\begin{document}

\newtheorem{lem}{Lemma}[section]
\newtheorem{prop}{Proposition}
\newtheorem{con}{Construction}[section]
\newtheorem{defi}{Definition}[section]
\newtheorem{coro}{Corollary}[section]
\newcommand{\hf}{\hat{f}}
\newtheorem{fact}{Fact}[section]
\newtheorem{theo}{Theorem}
\newcommand{\Br}{\Poin}
\newcommand{\Cr}{{\bf Cr}}
\newcommand{\dist}{{\rm dist}}
\newcommand{\diam}{\mbox{diam}\, }
\newcommand{\mod}{{\rm mod}\,}
\newcommand{\compose}{\circ}
\newcommand{\dbar}{\bar{\partial}}
\newcommand{\Def}[1]{{{\em #1}}}
\newcommand{\dx}[1]{\frac{\partial #1}{\partial x}}
\newcommand{\dy}[1]{\frac{\partial #1}{\partial y}}
\newcommand{\Res}[2]{{#1}\raisebox{-.4ex}{$\left|\,_{#2}\right.$}}
\newcommand{\sgn}{{\rm sgn}}

\newcommand{\CC}{\mathbb{C}}
\newcommand{\D}{{\bf D}}
\newcommand{\Dm}{{\bf D_-}}
\newcommand{\RR}{\mathbb{R}}
\newcommand{\NN}{\mathbb{N}}
\newcommand{\HH}{\mathbb{H}}
\newcommand{\ZZ}{\mathbb{Z}}

\newcommand{\tr}{\mbox{Tr}\,}
\newcommand{\R}{{\bf R}}
\newcommand{\C}{{\bf C}}

\newenvironment{nproof}[1]{\trivlist\item[\hskip \labelsep{\bf Proof{#1}.}]}
{\begin{flushright} $\square$\end{flushright}\endtrivlist}
\newenvironment{proof}{\begin{nproof}{}}{\end{nproof}}

\newenvironment{block}[1]{\trivlist\item[\hskip \labelsep{{#1}.}]}{\endtrivlist}
\newenvironment{definition}{\begin{block}{\bf Definition}}{\end{block}}

\newtheorem{conjec}{Conjecture}

\newtheorem{com}{Comment}
\font\mathfonta=msam10 at 11pt
\font\mathfontb=msbm10 at 11pt
\def\Bbb#1{\mbox{\mathfontb #1}}
\def\lesssim{\mbox{\mathfonta.}}
\def\suppset{\mbox{\mathfonta{c}}}
\def\subbset{\mbox{\mathfonta{b}}}
\def\grtsim{\mbox{\mathfonta\&}}
\def\gtrsim{\mbox{\mathfonta\&}}

\newcommand{\Poin}{{\bf Poin}}
\newcommand{\Bo}{\Box^{n}_{i}}
\newcommand{\Di}{{\cal D}}
\newcommand{\gd}{{\underline \gamma}}
\newcommand{\gu}{{\underline g }}
\newcommand{\ce}{\mbox{III}}
\newcommand{\be}{\mbox{II}}
\newcommand{\F}{\cal{F}}
\newcommand{\Ci}{\bf{C}}
\newcommand{\ai}{\mbox{I}}
\newcommand{\dupap}{\partial^{+}}
\newcommand{\dm}{\partial^{-}}
\newenvironment{note}{\begin{sc}{\bf Note}}{\end{sc}}
\newenvironment{notes}{\begin{sc}{\bf Notes}\ \par\begin{enumerate}}%
{\end{enumerate}\end{sc}}
\newenvironment{sol}
{{\bf Solution:}\newline}{\begin{flushright}
{\bf QED}\end{flushright}}

\title{Common Limits of Fibonacci Circle Maps}

\author{Genadi Levin\\
\small{Einstein Institute of Mathematics}\\
\small{Hebrew University}\\
\small{Givat Ram 91904, Jerusalem, ISRAEL}\\
\small{\tt levin@math.huji.ac.il}\\
\and
Grzegorz \'{S}wia\c\negthinspace tek\\
\small{Department. of Mathematics and Information Science}\\
\small{Politechnika Warszawska}\\
\small{Plac Politechniki 1}\\
\small{00-661 Warszawa, POLAND}\\
\small{\tt g.swiatek@mini.pw.edu.pl}
}
\normalsize
\maketitle

\abstract{We show that limits for the critical exponent tending to $\infty$ exist in both critical circle homeomorphism of golden mean rotation number and Fibonacci circle
coverings. Moreover, they are the same. The limit map is not analytic at the critical point, which is flat, but has non-trivial complex dynamics.} 

\section{Introduction}
\subsection{Limits of renormalization schemes.}
Since the seminal work of Feigenbaum several renormalization schemes have 
been discovered in the dynamics of maps in dimension $1$. They all share 
some common features: 
first return of a mapping from some topological class to suitably chosen intervals that shrink to the critical point (or value) of the dynamics when rescaled 
tend to a limit. This limit depends on the topological class of the original mapping and the exponent of the critical point, but is generally universal. Another 
key feature is that the limit is a solution to some functional equation which gives its rescaling to a smaller interval in term of the fixed point and its rescalings to 
a larger scale. Concrete examples other than the Feigenbaum class include the critical circle homeomorphisms, Fibonacci unimodal maps and Fibonacci circle coverings. 

Later, the dependence of these fixed points on the critical exponent 
was also examined in the situation when the exponent tends to $\infty$, 
or sometimes to $1$. 
So far, it has been done for the Feigenbaum class~\cite{EW},
\cite{feig} and golden mean 
critical circle mappings~\cite{expcirc}. 
It turned out the limits exist, or at least there is a strong evidence for their 
existence, and they were of the same topological type and satisfied the same functional equation as the fixed point maps for finite values of the exponent.  

In this paper we study Fibonacci circle coverings and homeomorphisms. They have been known to have fixed points of renormalization for odd integer values of the critical 
exponent, which belonged to certain distinct topological classes. For circle homeomorphisms the class that was generally used consisted of commuting pairs of maps on the 
interval. For coverings, the limits are box mappings with two branches. 
We have found that in both classes the fixed points 
tend to some limit class when the critical exponent grows to $\infty$. What is surprising is that it is the same limit class in both cases. Topologically, the limit map 
can be viewed as a circle homeomorphism with the golden mean rotation number. 
Thus, when one starts in the class of critical circle covers, in the 
limit of infinite 
criticality there is a change of topological type: for example the degree of the map drops from $2$ to $1$. 

Since this change was unexpected, one is left to wonder what may have caused it and when a similar phenomenon can be expected in the future. We can phrase our expectation 
in the form of a conjecture. 

\begin{conjec}\label{conjec:13xp,1}
Suppose that in a topological class of one-dimensional maps, 
each with one critical value, limits of appropriate 
renormalization schemes exist and satisfy a
functional fixed-point equation,
for each critical exponent from
a sequence tending to $\infty$.
Then there exists a topologically different class of one-dimensional maps,
for which limits of similar renormalization schemes exist
and satisfy the same functional equation. Furthermore, 
in the limit as the exponent tends to $\infty$ in both classes,
the fixed point maps tend to a common limit dynamics.

\end{conjec}

The evidence in favor of this conjecture has so far been rather scant, so perhaps it should rather be viewed as an open problem which in our opinion is worth further study. 
The functional equation is the same for critical circle homeomorphisms and Fibonacci covers where we have found a common universal class. It is different for Fibonacci 
unimodal maps, and indeed although the limits for criticality tending to $\infty$ have not been studied in detail for this class, one can already say that they will have 
two branches and thus be different from any other known limits class. A good test case for this conjecture would be to find a renormalization scheme leading to fixed points
which satisfy the Feigenbaum equation, but in a topological class different from infinitely renormalizable unimodal maps.  Another argument in support of the Conjecture 
follows from a discussion of the associate dynamics given in section~\ref{sec:14xp,1}.

\subsection{Statement of the results.}
\paragraph{Brief summary of the work.}
The methods used here and the general flavor of the results are the same as in an earlier work concerning the limits with infinite criticality for Feigenbaum maps. 
We build complex continuations of the dynamics and show that as the critical exponent tends to $\infty$ they lead to limits in a certain limit class of maps with flat 
critical points. This already turns out to be the same class for critical circle homeomorphisms and covers. In spite of not being analytic in a neighborhood of the critical 
point, this limit class shown to have a non-trivial complex dynamics. Out 
of this complex dynamics one can build McMullen towers and show their 
rigidity using the usual ideas. 
The rigidity of the towers implies the rigidity of the limits class itself which is shown to consist of just one mapping after a normalization. Out of this we derive 
the main theorem of the paper (Theorem~\ref{theo:12xp,1}) which does not mention directly the fixed points of renormalization, but is applicable directly to the underlying 
dynamics.

\paragraph{The classes of dynamics.}
\begin{defi}\label{defi:13xa,1}
Consider open intervals $I^0, I^{-1}, I$ in the following configuration: $\overline{I}^0 \cap \overline{I}^{-1} = \emptyset$, $\overline{I}^0 \cup \overline{I}^{-1} \subset I$,
$0\in I^0$. 
The branches $\psi^0, \psi^{-1}$ are defined on the corresponding intervals, are both monotone increasing, $C^3$, and map onto $I$. Furthermore, each branch $\psi^i$, $i=0,-1$,
has exactly one critical point with the local representation as 
$(\xi^i(x))^{\ell}$ where $\xi^i$ are diffeomorphisms and $\ell$ an odd integer bigger then $1$. This implies that both branches have the same critical value at $0$. 
The set of such mappings will be called ${\cal G}_{\ell}$.    
\end{defi}

Maps from ${\cal G}_{\ell}$ can be obtained naturally by inducing from critical circle coverings, see~\cite{lsuniv}. However, to use the results of that, and most other 
papers, one should consider branches $g^i(x) = \sqrt[\ell]{\psi^i(x^{\ell})}$. The effect of this change of normalization is that there is only one critical point at $x=0$.  

\begin{defi}\label{defi:13xa,2}
Consider open intervals $I^0$ and $I^{-1}$ which share a common endpoint and $I$ which is their convex hull. Suppose that $0\in I^0$. There are branches $\psi^{0},\psi^{-1}$ 
defined from the respective intervals into $I$, both increasing, and after identifying the endpoints of $I$, the union of branches $\psi^0$ and $\psi^{-1}$ extends to a circle 
homeomorphism. 

Each branch $\psi^i$ is a diffeomorphisms of class $C^3$ and has a $C^3$ extension to a neighborhood of the common endpoint $a^i$ of $I$ and $I^i$ in the form  
$\psi^i(x) = (\xi^i(x))^{\ell}$ where $\xi^i$ is a local diffeomorphism  and $\xi^i(a^i) = 0$. The class of such mappings with be denoted with ${\cal G}_{\ell}^1$. 
\end{defi}

\paragraph{Fibonacci combinatorics.}
Our use of the term Fibonacci combinatorics is somewhat different for maps of degree $1$ and higher. For homeomorphisms, we simply mean that the rotation  
number is the golden mean. 

A map $g$ from ${\cal G}_{\ell}$ is said  
to have the Fibonacci combinatorics if there exists a weakly order preserving map from the circle to $I$ with endpoints identified, which conjugates $g$ on the forward orbit 
of the critical point to the dynamics of  an orbit under the golden mean rotation.

\subparagraph{Renormalization.}
Suppose that $\psi$ is either a circle homeomorphism or a map from ${\cal G}_{\ell}$ with Fibonacci combinatorics. 
We normalize $\psi$ so that the critical value is at $0$. Then consider the sequence $x_{q_n}$ where $x_{q_n}$ is the preimage of $0$ of order $q_n$, closest to 
$0$ and $q_n$ is the $n$-th Fibonacci number. Let $I_n$ denote the interval between $x_{q_n}$ and $x_{q_{n-1}}$. Let $\psi_n$ denote the first return map of $\psi$ into $I_n$. 
In particular, $\psi_n(0)=\psi^{q_n}(0)$. 

Let $\zeta_n$ be a linear map specified by the condition 
$\zeta_n(\psi^{q_n}(0)) = 1$. Then, the $n$-th renormalization, 
\[ {\cal R}^n(\psi) = \zeta_n \circ \psi_n \circ \zeta^{-1}_n \]
is a map defined from a dense and open subset of
$\zeta_n(I_n)$ into $\zeta_n(I_n)$.

From Theorem 1 in~\cite{lsuniv} we get the following:

\begin{fact}\label{fa:13xp,1}
For each $\ell$ which is an odd integer greater than $1$, there exists exactly one map $H_{\ell} \in {\cal G}_{\ell}$ with the Fibonacci combinatorics and constant 
$\tau_{\ell} < -1$ so that if $\phi=\phi^0, \phi^{-1}$ are the branches of $H_{\ell}$, then 
\begin{itemize}
\item $\phi(0) = 1$,
\item \[    \phi^{-1}(x) = \phi_{-1}(x)  \] for  $x\in I^{-1}$ where the subscript $k$ denotes a rescaling by $\tau^{-k}_{\ell}$,
e.g., $\phi_{-1}(x) = \tau \circ \phi\circ \tau^{-1}$, 
\item the {\em fixed point equation} holds for all  $x \in \tau^{-1}_{\ell}I^0$:
\begin{equation}\label{equ:27up,1} 
\phi_1(x) = \phi_{-1}\circ \phi(x)\; .
\end{equation}
\item
for any $\psi\in {\cal G}_{\ell}$ with the Fibonacci combinatorics, 
the sequence of renormalizations ${\cal R}^n(\psi)$ converges to $H_{\ell}$ 
uniformly on the domain of $H_{\ell}$. 
\end{itemize} 
\end{fact}

Similarly, from~\cite{defaria} and~\cite{demelofaria} we get this:
\begin{fact}\label{fa:13xp,2}
For each $\ell$ which is an odd integer greater than $1$, there exists exactly one map $H^1_{\ell} \in {\cal G}^1_{\ell}$ with the Fibonacci combinatorics and constant 
$\tilde{\tau}_{\ell} < -1$ so that if $\phi=\phi^0, \phi^{-1}$ are the branches of $H^1_{\ell}$, then 
\begin{itemize}
\item $\phi(0) = 1$,
\item \[    \phi^{-1}(x) = \phi_{-1}(x)  \] for  $x\in I^{-1}$ where the subscript $k$ denotes a rescaling by $\tilde{\tau}^{-k}_{\ell}$, 
\item the fixed point equation~(\ref{equ:27up,1}) holds for all  $x \in 
\tilde\tau^{-1}_{\ell}I^0$,
\item
for any $\psi\in {\cal G}^1_{\ell}$ with the Fibonacci combinatorics, 
the sequence of renormalizations ${\cal R}^n(\psi)$ converges to $H^1_{\ell}$ uniformly on the domain of $H^1_{\ell}$. 
\end{itemize} 
\end{fact}

\paragraph{The main theorem and its corollaries.}
\begin{theo}\label{theo:12xp,1}
There exist $x_0<0$ and $\tau<-1$ for which the following holds.
Consider a sequence of all odd integers $\ell$. For each $\ell$ consider a map $\psi_{\ell}$ which is either in ${\cal G}_{\ell}$ or ${\cal G}^1_{\ell}$ with the Fibonacci 
combinatorics.  Let ${\cal R}^n(\psi_{\ell})$ be the sequence of renormalizations.  
Then, there exists a sequence $k_{\ell}$ such that for any sequence $n_{\ell} \geq k_{\ell}$ for all $\ell$, mappings ${\cal R}^{n_{\ell}}(\psi_{\ell})$ converge almost uniformly 
on the set $(x_0,x_0/\tau) \cup (x_0/\tau,x_0\tau)$ as $\ell \rightarrow\infty$. This limit is independent of the sequence $\psi_{\ell}$ and is a homeomorphism of the 
circle obtained by identifying $x_0$ with $\tau x_0$ with the golden mean rotation number. 
It further belongs the ${\cal EWF}$-class defined later, see Definition~\ref{defi:9xa,1}.
\end{theo} 

\begin{coro}\label{coro:13xp,1}
The scaling factors $\tau_{\ell}$ and $\hat{\tau}_{\ell}$ introduced in Facts~\ref{fa:13xp,1} and~\ref{fa:13xp,2} tend to a common limit $\tau<-1$. 
\end{coro}

The limit $\tilde\tau_\ell\to \tau$ as $\ell\to \infty$ was the subject
of an experimental study, see~\cite{expcirc}. In particular, 
numerically, $\tau=-3.71...$.

The next result does not follow formally from the main theorem, but can be derived from its proof.
\begin{theo}\label{theo:13xp,4}
The Hausdorff dimension of the post-critical set for maps in ${\cal G}_{\ell}$
with the Fibonacci combinatorics, which depends only on $\ell$ by~\cite{lsuniv}, tends to $1$ as $\ell$ tends to $\infty$. 
\end{theo}

Note that for Feigenbaum maps of the interval, the Hausdorff dimension of the attractor tends to a limits which is less than
$1$, see~\cite{leprzy}. It is related to the fact that in the Feigenbaum case the topological dynamics does not change in the
limit as $\ell\to\infty$, but in the case of Fibonacci covers studied in the present paper, such a change occurs. Further
comments on this phenomenon follow in connection with the associated dynamics of $G$. 

In the complex plane, this difference disappears and the Hausdorff dimensions of the Julia sets for Feigenbaum maps were shown
to tend to $2$ in~\cite{LSw2}. For Fibonacci covers there is no proof in the literature, but we expect the arguments
of~\cite{LSw2} to work with minor changes.

\subsection{The limit class ${\cal EWF}$.}\label{sec:14xp,1}
Let us first recall the concept of Poincar\'{e} neighborhoods:
\begin{defi}\label{defi:27up,1}
If $I$ is an open interval, then ${\cal D}(I)$ denotes the geometric disk centered at the midpoint of $I$. Similarly, for $0<\alpha\leq \pi$, we write 
${\cal D}(I,\alpha)$ to denote the set of point $z$ in the plane such that the circle passing through $z$ and the endpoints of $I$ intersects the real
line at angle less than $\alpha$ (small $\alpha$ meaning a small section of the disk). 
\end{defi}

In particular, ${\cal D}(I,\pi)$  is the doubly slit plane $(\CC\setminus\RR)\cup I$. 

A star added to the notation of a disk, or disk neighborhood defined
above, will mean a punctured neighborhood with the point $0$
removed. For example, ${\cal D}^*(-1,2)$ is equivalent to ${\cal
  D}(-1,2)\setminus \{0\}$.

\begin{defi}\label{defi:9xa,1}
We start by specifying the map on the real line. 

Fix parameters $x_0<0$, $\tau<-1$ and $\tau^2 > R>\tau x_0$, assuming also $\frac{x_0}{\tau} < 1 < \tau x_0$. 
We consider a mapping $\phi$ continuous on an interval $[x_0,R')$, $R'>x_0/\tau$, and a real-analytic orientation-preserving   
diffeomorphism from $(x_0,R')$ onto its image $(0,R)$. Suppose that $\phi(x_0) = 0, \phi(0) = 1$, $\phi(x_0/\tau) = \tau x_0$ and $\phi(x_0/\tau^2)=x_0/\tau$.  

We can also consider mappings $\phi_{k} = \tau^{-k} \phi \tau^{k}$ for $k\in\ZZ$.  One can easily see that $\phi_{-1}(x_0\tau)=0$ and 
$\phi_{-1}(x_0/\tau) = \tau  \phi(x_0/\tau^2) = x_0$. 
After identifying points $x_0$ and $\tau x_0$ and putting $\phi$ on $[x_0,x_0/\tau]$ and $\phi_{-1}$ on $[x_0/\tau,x_0\tau]$, we get a degree $1$ circle homeomorphism $\cal H$.

We assume that the rotation number of $\cal H$ is the golden mean, $\frac{1}{1+\frac{1}{1+\cdots}}$ and that functional equation~(\ref{equ:27up,1}) holds for 
$x\in [x_0/\tau^2, x_0/\tau]$.

Furthermore, $\phi$ has an analytic continuation, also denoted by $\phi$, and we make the following assumptions about it:

\begin{enumerate}    
\item
$\phi$ is defined on a topological disk $U$ and $U\cap\RR= (x_0,R')$. Also, $U$   is  symmetric with respect to the real axis.  
\item
$\phi$ is a covering
(unbranched) of the punctured disk $V := D(0, R) \setminus \{0\}$ by $U$.  
\item
For any $0<r\leq R$, $\phi^{-1}(D(0,r) \setminus \{0\})$ is contained in ${\cal D} (x_0,r')$ where $r'$ is real and $\phi(r')=r$. 
It implies that $U\subset {\cal D}(x_0,R')$. 
\item
If $I$ is a real segment which does not contain $0$ and $\phi^{-1}$ denotes the inverse branch of $\phi$ which maps $I$ into $\RR$, then for any $0<\alpha\leq\pi$
$\phi^{-1}\left( {\cal D}(I,\alpha)\right) \subset {\cal D}(\phi^{-1}(I),\alpha)$. 
\item
Define the mapping $G(x) := \tau ^{-1} \phi \tau^{-1}$. By previous hypotheses, it  fixes $x_0$ and is analytic in its neighborhood.  
Assume that $G$ has the following power
series expansion at $x_0$:
\[ G(x) = x - \epsilon (x-x_0)^3 + O(|x-x_0|^4) \]
with $\epsilon> 0$.  
\end{enumerate}

The class of all mappings $\phi$ with these properties will be denoted with 
${\cal EWF}$.  
\end{defi}

\begin{figure}
\epsfig{figure=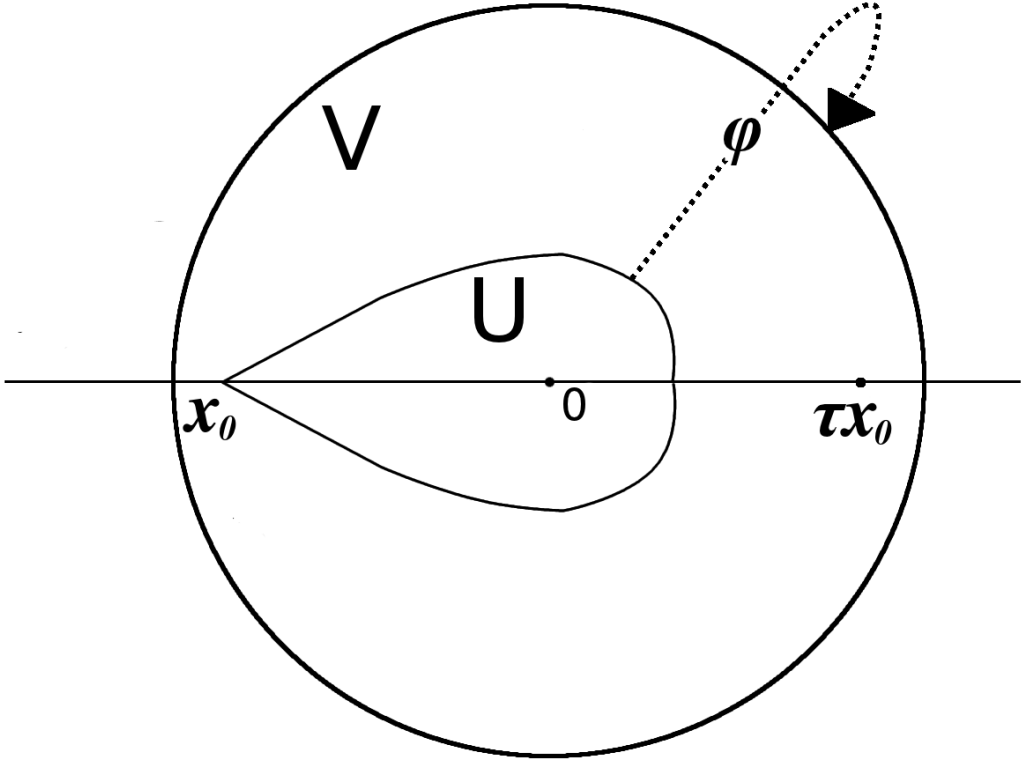, height=12cm, width=15cm}
\caption{The mapping $\phi$.}
\end{figure}

\begin{theo}\label{theo:13xp,3}
Class ${\cal EWF}$ consists of one mapping.
\end{theo}

\paragraph{Associated dynamics of $G$.}
An insight into the nature of the flat critical point of $\phi$, but also the bifurcation which occurs for the limit dynamics and allows one to unfold it 
as either a circle homeomorphism, or a covering map, can be gained from looking at the associated dynamics of the function $G$. Function $G$ appears in item 5. of 
Definition~\ref{defi:9xa,1}. Its dynamical interpretation comes from the functional equation~(\ref{equ:27up,1}) and is stated as Lemma~\ref{lem:29up,2}, or simply:
\begin{equation}\label{equ:14xp,1} 
\phi\circ G = \tau^{-2} \phi \; .
\end{equation}

Since $x_0$ is a neutral but topologically attracting fixed point of $G$, equation~(\ref{equ:14xp,1}) implies that it has to be a flat critical point of $\phi$. Further 
information about this point can be gained from the interpretation of $\log\phi$ as a Fatou coordinate for this point. Perhaps more interestingly, one can consider 
the bifurcation of this fixed point which corresponds to considering the dynamics of fixed points of renormalization $H_{\ell}$ or $H^1_{\ell}$ for $\ell$ large but 
finite. There are two ways of bifurcating a neutral, topologically attracting fixed point for a real analytic map with negative Schwarzian derivative. 
One will create a pair of attracting fixed points and a repelling one  between them, all on the real line. The other is to make the fixed
point on the real line attracting, and form a pair of repelling ones in the complex plane. 

The second bifurcation regime is easier to understand since it leads to no topological change of the dynamics on the real line. 
 The first mode creates a repelling fixed point on the real line. By equation~(\ref{covphiG}), a repelling fixed point of
 $G_{\ell}$ corresponds to a singularity where $H_{\ell}$ goes to $\infty$, see also Lemma~\ref{covcontin}. Hence, the
 dynamics between the attracting points has unbounded image for every $\ell$ even though its domain shrinks to a point 
as $\ell$ goes to $\infty$. The dynamics of $H_{\ell}$ outside the interval between the attracting fixed points undergoes no
bifurcation at the limit. However, the dynamics between the fixed points vanishes in the limit and this is the reason why the
topological degree drops.       

As an example of the consequences of this, we can recall the difference in limiting behavior of the Hausdorff dimension of the
attractors for Feigenbaum polynomials and circle covers. For Feigenbaum polynomials the second type bifurcation occurs which
leads to no change of dynamics of the real line as a consequence no qualitative change of the Hausdorff dimension. As will  be
show in this paper, for circle coverings the first type of bifurcation happens which leads to the disappearance of a part of
dynamics and closing of gaps on Cantor sets in the limit. 

Finally, the big difference between two types of bifurcation is only relevant on the real line. So, one does not expect to see
it when studying the Hausdorff dimension of Julia sets, for example.

\subsection{Basic properties of ${\cal EWF}$ maps.}

\paragraph{Connection between $\phi$ and the Fatou coordinate.}
Observe that $\log \phi$ is a well defined univalent map. Its inverse
can be defined as the lifting of $\exp$ to the universal covering
$\phi$. 
\begin{lem}\label{lem:29up,2}
The transformation $h(z) = \frac{\log\phi}{\log \tau^{-2}}$ is a Fatou
coordinate for $G$:
\[ h\circ  G(z) = h(z)+1 \]
for all $z\in U$.
\end{lem}

This provides useful information about $h$, and therefore the
singularity of $\phi$ at $x_0$, in the light of uniqueness of the
Fatou coordinate.  It leads to the next geometrical lemma.

\begin{lem}\label{lem:28ua,1}
For any $\epsilon>0$ there is $\delta>0$ and for each $z\in U$, if
$|z-x_0|<\delta$, then $|\arg(z-x_0)|<\frac{\pi}{4}+\epsilon$. 
\end{lem}
\begin{proof}
The key to the proof of this Lemma is Lemma~\ref{lem:29up,2} and the transformation $h$ introduced there. By its definition, the range of $h$ is contained in a certain 
right half-plane $\Re w > A_1$. Next, we consider a standard construction of the Fatou coordinate following for example~\cite{carleson}, which gives 
$\tilde{h} = \xi(\frac{C}{(z-x_0)^2})$ where $\xi(z) = z + O(|z|^{-1/2})$ and $C>0$. Since a relevant inverse branch of $\frac{C}{(z-x_0)^2}$ maps the set 
$\{ w :\: \Re w > 0\}$ inside an angle $|\arg(z-x_0)|<\frac{\pi}{4}$, the preimage $\tilde{h}^{-1}(\{ w :\: \Re w > A_2\}$ is also contained in the same angle 
if $A_2$ is chosen sufficiently large. 

By the uniqueness of the Fatou coordinate $h(z) - \tilde{h}(z) = T$ and so $W:=h^{-1}(\{ w :\: \Re w > A_2 + T\})$ is contained in the same angle. Finally, choose $k$ so that 
$k> A_2+T-A_1$. Then, by Lemma~\ref{lem:29up,2}, the domain $U$ is contained in $G^{-k}(W)$. Since $G$ is conformal, the claim follows.
\end{proof}   

Based on the interpretation of $\log\phi$ as a Fatou coordinate up to a normalization, see~\cite{carleson}, one gets:
\begin{fact}\label{fa:29up,3}
$\log\phi$ extends to a quasiconformal mapping of the Riemann sphere
 sending $x_0$ to $\infty$.
\end{fact}

\paragraph{Analytic continuation of the functional equation.}
\begin{lem}\label{lem:9xa,1}
Suppose that $\phi$ a mapping from the ${\cal EWF}$-class. Then the functional equation~(\ref{equ:27up,1}) holds for every argument $z\in \tau^{-1} U$, meaning also that both sides of the functional equation are well defined. 
\end{lem}
\begin{proof}
We will show a topological disk $W$ such that $W \cap \RR = (a, x_0/\tau)$, $\phi_{-1}\circ\phi(a)=R/\tau$, and $\phi_{-1} \circ \phi$ is defined on $W$ and a covering of $D(0,R/\tau)\setminus \{0\}$. 
By item 3 of Definition~\ref{defi:9xa,1}, $\phi_{-1}$ provides a universal covering of  $D(0,R/\tau)\setminus \{0\}$ by some topological disk $W_1$ such that 
$W_1 \cap \RR = (a_1, \tau x_0)$, $W_1 \subset {\cal D}(a_1,\tau x_0)$ and 
\[ a_1=\phi_{-1}^{-1}(R/\tau) > \phi_{-1}^{-1} \tau = \tau \phi^{-1}(1) = 0 \; .\]
Then $W = \phi^{-1}(W_1)$ where $\phi^{-1}$ denotes a univalent inverse branch which maps the real trace $(a_1,\tau x_0)$ to its preimage $(a,x_0/\tau)$. 

Then looking at $\phi_1$ on its domain $\tau^{-1} U$, we observe that it is a universal covering of $D(0,R/\tau)\setminus\{0\}$. So $\phi_1$ and $\phi_{-1}\circ\phi$ are 
universal coverings of the same set and are equal on a segment $(x_0/\tau^2,x_0/\tau)$ by equation~(\ref{equ:27up,1}). Then the lifting of the identity is a univalent map from 
$W$ onto $U$ which is the identity on the segment and thus globally. 
\end{proof}

\section{Critical circle covers}
This section is devoted to the proof of the following theorem.
\begin{theo}\label{theo:13xp,1}
Take a sequence of odd integers $\ell_n$ tending to $\infty$. 
Consider a sequence of fixed point maps $H_{\ell_n}$ and scaling constants $\tau_{\ell_n}$ introduced in Fact~\ref{fa:13xp,1} and let $\phi_{\ell_n}$ denote their branches 
whose domains contain $0$.  
Let $x_n$ denote the critical point of the $\phi_{\ell_n}$. 
For a subsequence $n_k$ the following are true:
\begin{itemize}
\item
sequences $x_{n_k}$ and $\tau_{\ell_{n_k}}$ converge to $x_0$ and $\tau$, respectively, which satisfy the inequalities postulated by 
Definition~\ref{defi:9xa,1}, 
\item
mappings $\phi_{\ell_{n_k}}$ converge almost uniformly on $(x_0,x_0/\tau)$ to $\phi$ which belongs to the ${\cal EWF}$-class.
\end{itemize}
\end{theo}

\subsection{Fixed-point equations}

\begin{figure}
\epsfig{figure=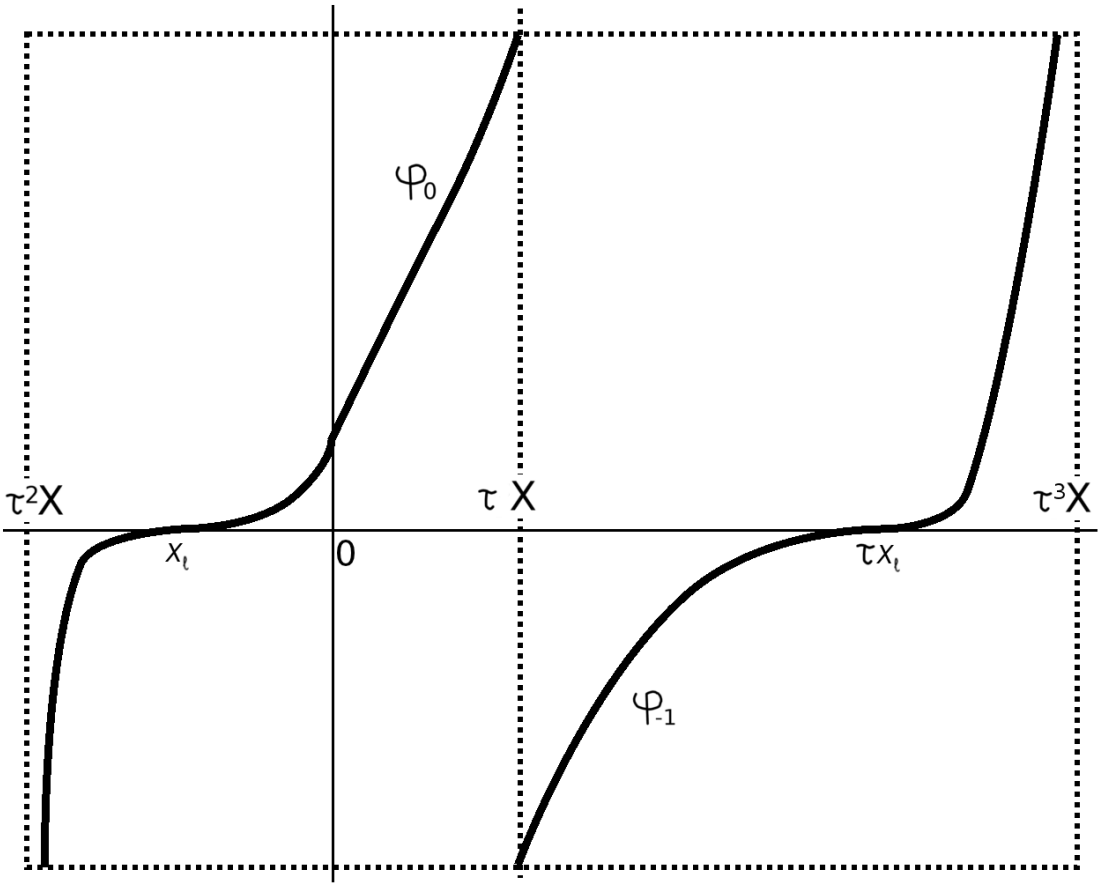, height=12cm, width=15cm}
\caption{The graph of $H_{\ell}$.}
\end{figure}

\paragraph{Map near the critical point}
Fix an odd integer $\ell\ge 3$, and let $H_\ell$ be the map from the 
Fact~\ref{fa:13xp,1}. Consider also the corresponding map $h_\ell$ near
the critical point. In other words, if $p(x)=x^\ell$ is the change of variable
on $\RR$,
then $h_\ell=p^{-1}\circ H_\ell\circ p$. It consists of two branches
$g_i=p^{-1}\circ \phi_i\circ p$, where $\phi_i$, $i=0,-1$ are the
branches of $H_\ell$, so that $g_i$ is defined on $J^i=p^{-1}(I^i)$
with the common image $J=p^{-1}(I)$.
The scaling factor for $h_\ell$ is
$\alpha_\ell=\tau_\ell^{1/\ell}<-1$.
The first return map of $h_\ell$ to the central interval $J^0$
retained to those components in $J^0$,
which intersect the forward critical orbit $\{h^i(0)\}_{i\ge 0}$ 
consists of the central branch which is $g_{-1}\circ g_0$ and is defined
on $\alpha^{-1}J^0$, and the off-central branch which is
$g_0$ and is defined on $\alpha_\ell^{-1}J^1$. Furthermore,
after the rescaling by $x\mapsto \alpha_\ell x$, 
this first return map coincides with the
initial map $h_\ell$. In other words, $\alpha_\ell^{-1}J^1\subset J^0$ and 
$\alpha_\ell J^0=J$,
and $g_0$, $g_{-1}$ satisfy the functional equations
\begin{eqnarray}\label{covfixp0}
g_{-1} &=& \alpha\circ g_{0}\circ \alpha^{-1}, \\
g_{0}&=& \alpha\circ g_{-1}\circ g_{0}\circ \alpha^{-1}.
\end{eqnarray}
Since $g_0(x)=E(x^\ell)$, where a diffeomorphism $E=E_\ell$
belongs to the Epstein class,
it follows that the map $h_\ell$ has a non-positive Schwarzian derivative.
Recall that a real-analytic diffeomorphism belongs to the Epstein class provided that 
its inverse has an analytic continuation to the upper half-plane. Note that $E$ maps $I^0$ onto $J$. 

For every $n\ge 1$, consider the first return map $h^n$ of $h_\ell$
to the interval $J_n=\alpha_\ell^{-n}J$ (so that $h^1=h_\ell$ and $J_1=J^0$) 
retained to the components in $J_n$ intersecting the forward critical orbit.
As it follows from (\ref{covfixp0}), the map $h^n$
is equal to $\alpha_\ell^{-n}\circ h_\ell\circ \alpha_\ell^{n}$. 
$h^n$ consists of a pair of homeomorphisms
$g_n=\alpha_\ell^{-n}\circ g_0\circ \alpha_\ell^{n}:J_{n+1}\to J_n$
and $g_{n-1}=\alpha_\ell^{-n+1}\circ g_0\circ \alpha_\ell^{n-1}:
\alpha^{-n}J^1\to J^n$. 
By the Fibonacci combinatorics,
the first entry of the $h_\ell$-iterates of $0$
to the interval $J_n$ occurs at the time $q_n$,
where $q_0=1, q_1=2, q_2=q_0+q_1=3,...$ are the Fibonacci numbers.

\paragraph{Analytic continuation}
Fix $\ell$. In this paragraph, we sometimes drop
the index $\ell$ denoting $\tau=\tau_\ell$, $\alpha=\alpha_\ell$
etc. Let $\phi=\phi_0$, the branch of $H$ which contains $0$ in its
domain $I^0$. 
Recall that $\phi$ satisfies the equation
\begin{equation}\label{covphi}
\tau^{-2}\phi(x)=\phi\circ \tau^{-1}\circ \phi\circ \tau^{-1}(x), 
\ \ x\in I^0_c.
\end{equation}
We define 
\begin{equation}\label{equ:21za,1}
G_{\ell}=\tau_{\ell}^{-1}\circ \phi_{\ell}\circ \tau_{\ell}^{-1} .
\end{equation}
Then ~(\ref{covphi}) turns into
\begin{equation}\label{covphiG}
\tau_{\ell}^{-2}\phi_{\ell}=\phi_{\ell}\circ G_{\ell}.
\end{equation}
Denote by $y_\ell$ the unique zero of $g_0: J^0\to J$. 
Set $x_\ell=p(x_\ell)=y_\ell^\ell$. 
Then $x_\ell$ is the only zero and the only critical point of $\phi$
on $I^0$. Denote $I=(A, B)$.
Then $I^0=(B/\tau, A/\tau)$ and $B/\tau<x_\ell<0<A/\tau$.

\begin{lem}\label{covcontin}

(a) $\phi: (B/\tau, A/\tau)\to (A, B)$ has a real analytic continuation to  
$\phi: (A, B)\to \tau^2(A, B)$. 

(b) The restriction
$\phi: (x_\ell, \tau x_\ell)\to (0, \tau^2)$ is a diffeomorphism,
which belongs to the Epstein class.

(c) $\phi(1/\tau)=1/\tau^2$, $\phi'(1/\tau)=1$,
$\phi(x_\ell/\tau)=\tau x_\ell$.
Also, $1/|\tau|<|A|<|\tau|$,
$1<\tau x_\ell<B<|\tau|^2$.

(d) (associated dynamics of $G$).
There exists a unique point $X\in (1/\tau, 0)$, such that
$\phi(X)=\tau X$. Then $\tau \phi(x_\ell/\tau^2)<\tau^2 X<x_\ell$,
and these $3$ points are the only real
fixed points of $G$. Moreover, the points 
$\tau \phi(x_\ell/\tau^2)$ and $x_\ell$ are attracting with
the multiplier $1/\alpha^2$ while the point $\tau^2 X$ is strictly repelling.

(e) mapping $\phi$ extends to a real-analytic homeomorphism from $(\tau^2 X, \tau x_\ell)$
onto $(-\infty, \tau^2)$, such that $\phi=E^\ell$, where
$E: (\tau^2 X, \tau x_\ell)\to (-\infty, \alpha^2)$ is a diffeomorphism
from the Epstein class. 
\end{lem}
\begin{proof}
Map $G_{\ell}$, here abbreviated to $G$, is given by formula~(\ref{equ:21za,1}). It is real-analytic on $(A, B)$ and
$G((A, B))=(B/\tau, A/\tau)$. Then (a) follows from~(\ref{covphi}).
Since $\phi=(E)^\ell$ and $E$ is in the Epstein class, then
the inverse map $\phi^{-1}: (0, B)\to (x_\ell, A/\tau)$
extends to a univalent map 
$\phi^{-1}: {\CC}_{(0, B)}\to {\CC}_{(x_\ell, A/\tau)}$, and hence,
for a real branch of the map $G^{-1}$, 
there is a univalent extension
$G^{-1}: {\CC}_{(B/\tau, 0)}\to {\CC}_{(A, \tau x_\ell)}$.
Since $B/\tau<x_\ell$, 
it allows us to define 
a univalent map $\phi^{-1}=G^{-1}\circ \phi^{-1}\circ \tau^{-2}$
from ${\CC}_{(0, \tau^2)}$ into ${\CC}_{(x_\ell, \tau x_\ell)}$,
where it is a diffeomorphism on the real line.
Finally, let us show (c). We use (\ref{covphi}). First,
$\tau^{-2}=\tau^{-2}\phi(0)=\phi\circ \tau^{-1}\phi\circ \tau^{-1}(0)=
\phi(1/\tau)$. Besides, $\tau^{-2}\phi'(0)=\tau^{-2}\phi'(1/\tau)\phi'(0)$,
i.e., $\phi'(1/\tau)=1$. Also, $\phi(G(x_\ell))=\tau^{-2}\phi(x_\ell)=0$,
which implies that $G(x_\ell)=\tau^{-1}\phi(x_\ell/\tau)=x_\ell$.
To prove the rest of (c), notice that, by the combinatorics,
$B>\tau x_\ell>\phi(0)=1>A/\tau$. Coupled with the inequality 
$1/|\tau|<|A/B|<|\tau|$, this implies the inequalities of (c). 

(d) The function $\tau^{-1}\phi$ is strictly decreasing in a neighborhood of
$[1/\tau, 0]$, and $\tau^{-1}\phi(1/\tau)=1/\tau^3>1/\tau$, 
$\tau^{-1}\phi(0)=1/\tau<0$. Hence, it has a unique fixed point 
$X\in (1/\tau, 0)$. Introduce $\Gamma=\tau\circ \phi\circ \tau^{-2}$.
By the functional equation, $\Gamma^2=G$. But $\Gamma(\tau^2 X)=\tau^2 X$.
Therefore, $\Gamma(x_\ell), \tau^2 X, x_\ell$ are fixed points
of $G$. Let us show that $\tau \phi(x_\ell/\tau^2)<\tau^2 X<x_\ell$.
Indeed, as for the map $h$, the points $y_\ell/\alpha_\ell^2$
and $y_\ell/\alpha_\ell$ lie in the central interval $I^0$. Hence
$y_\ell/\alpha_\ell<g_0(y_\ell/\alpha_\ell^2)\in I^1$.
Then $x_\ell/\tau< \phi(x_\ell/\tau^2)$, i.e. $x_\ell<\Gamma(x_\ell)$.
In turn, since $\Gamma$ is strictly decreasing,
this implies that $\Gamma(x_\ell)<\tau^2 X<x_\ell$.
Furthermore, since $\phi(x_\ell+x)=a x^\ell + o(|x|^{\ell})$,
expanding $G$ into the power series and substituting in 
(\ref{covphi}) yields $G'(x_\ell)=G'(\Gamma(x_\ell))=\alpha^{-2}$. 
That is, $x_\ell, \Gamma(x_\ell)$ are attracting fixed points 
of $G$. On the other hand, $G$ is in the Epstein class,
in particular, $SG\le 0$. It follows that $\tau^2 X$ is the only other fixed
point of $G$, and it is strictly repelling.

(e) For any $x\in (\tau^2 X, x_\ell)$, we find $n\ge 0$ such that
$G^n(x)\in (B/\tau, x_\ell)$. Then we can define
$\phi(x)=\tau^{2n}\phi(G^n(x))$. Clearly,
$n\to \infty$ as $x\to \tau^2 X$, i.e. $\phi(x)\to -\infty$ as $x\to \tau^2 X$.
It follows that $E$ in the decomposition $\phi=E^\ell$ extends
to a real-analytic diffeomorphism 
$E: (X, \tau x_\ell)\to (-\infty, \alpha^2)$.
It remains to check that this extension is still in the Epstein class.
Indeed, we have an univalent map 
$E^{-1}: \CC_{(A^{1/\ell}, B^{1/\ell})}\to \CC_{(B/\tau, A/\tau)}$.
By the above, $E^{-1}: (A^{1/\ell}, B^{1/\ell})\to (B/\tau, A/\tau)$
extends to a real-analytic diffeomorphism from
$(-\infty, \alpha^2)$ onto $(X, \tau x_\ell)$. By
the Uniqueness Theorem for analytic functions, we are done.
\end{proof}

\subsection{Bounds for covering maps} 
\paragraph{Real bounds.}

\begin{prop}\label{covrb}
There exist
two constants $1<T_1<T_2<\infty$,
such that $T_1<|\tau_\ell|<T_2$, for all $H_\ell$.
\end{prop}

The proof is contained in the following two lemmas~\ref{covT_1},
~\ref{covT2}.

\begin{lem}\label{covT_1}
There exists $1<T_1$,
such that $T_1<|\tau_\ell|$ for all $H_{\ell}$.
\end{lem}
\begin{proof} 
Fix $\ell$ large. 
It is enough to prove that there is $C>0$ independent on $\ell$
such that $|\alpha_\ell|\ge 1+C/\ell$. We will drop the index $\ell$ in
$\alpha_\ell, \tau_\ell$, $h_\ell$ etc.

Consider the first entry map $h_{n-1}$ to the interval $J_{n-1}$, for
$n$ large enough. Let us apply the shortest interval argument
to the collection of pairwise disjoint intervals 
$J^j_n=h^j(J_n)$, $j=0,1,...,q_{n-1}-1$.
Denote by $J^k_n$ the shortest interval of this collection.
Consider two cases.

1. $k=0$, i.e., $J_n$ is the shortest interval. 
The intervals $h^{q_{n-2}}(J_n)$ and $h^{q_{n-3}}(J_n)$
lie on the opposite sides from $J_n$ and both are subsets of $J_{n-3}$.
Therefore, by the assumption, $|J_{n-3}|\ge 3|J_n|$.
On the other hand, $|J_{n-3}|/|J_n|=|\alpha|^3$, hence,
$|\alpha|\ge 3^{1/3}$.

2. $k\ge 1$. Denote $J=(a,b)$. As $\alpha$ is negative and
$J_{1}=\alpha^{-1}J\subset J$, then $|\alpha|^{-1}\le |a|/b\le |\alpha|$.
Since $J_n$ is the interval between the points
$a/\alpha^n, b/\alpha^n$, 
the origin divides the interval
$J_n$ into two intervals with the ratio of their lengths
at most $|\alpha|$. Now, we can assume that $|\alpha|<1+1/\ell$
because otherwise there is nothing to prove. Therefore, 
the above ratio is at most $1+1/\ell$. Consider
$J_n^1=h(J_n)=E(p(J_n))$, where, as before, $p(x)=x^\ell$,
and $E=E_\ell$ is a diffeomorphism near $0$, $E(0)=1$. 
Hence, given $\ell$, $|E'(x)/E'(y)|<2$ for
$x,y\in J_n$ provided $n$ is large. 
We get
the following {\it fact}: there 
is $C_0>1$ independent on $\ell$, such that,
provided $n$ is large enough, the point $1$
divides the interval $J_n^1$ into two intervals with the ratio
of their lengths at most $C_0$. 
Denote by $J_n^{j_1}, J_n^{j_2}$ two neighbors
of $J_n^k$ from the right and from the left
(if there is no neighbor from one side,
add an interval on this side of the length $|J_n^k|$). 
Denote by $K$ the smallest interval containing $J_n^{j_1}, J_n^{j_2}$.
Let us pull back $K$ by a branch of $h^{k-1}$ from $J_n^k$ 
to $J_n^1$. Along the way, a preimage of either $J_n^{j_1}$
or $J_n^{j_2}$ can turn into $I_n^1$ at most once.
If it happens, i.e., say, for $i=j_1-1$, 
$h^{-i}(J_n^{j_1})=J_n^1$, the point $1$ splits $h^{-i}(K)$ into
two parts. Then 
we cut off the part of $h^{-i}(K)$ which does not contain $h^{-i}(J_n^k)$,
and continue to pull back.
Since $h$ has a non-positive Schwarzian and by the above {\it fact},
there exists some $C_1>0$ independent on $\ell$, such that, for
every $n$ large enough, a $C_1$-neighborhood 
$K_1=\{x: dist(x, J_n^1)<C_1|J_n^1|\}$ of the interval $J_n^1$
contains at most one interval on each side of $J_n^1$ from the collection
of intervals $J^j_n$, $j=0,1,...,q_{n-1}-1$. As
$J_{n-5}$ contains at least two intervals from our collection
on each side of $J_n$,
$K_1$ is contained in $h(J_{n-5})$. Now we use that $n$ is large and get
$3 C_1\le |h(J_{n-5})|/|h(J_n)|\le 2 |\tau|^5$, i.e.,
$|\tau|^5\ge C_2$, where $C_2=3 C_1/2$ is independent on $\ell$. 
\end{proof}

\begin{lem}\label{covT2}
There exists 
$T_2$ such that for all $H_{\ell}$, 
we get $|\tau_{\ell}| < T_2$. 
\end{lem}
\begin{proof} 
Consider the map $\phi_{-1}=\tau\circ \phi\circ \tau^{-1}$.
By Lemma~\ref{covcontin}, 
$\phi_{-1}: (\tau ^2 x_\ell, \tau x_\ell)\to (\tau^3, 0)$ is a diffeomorphism
from the Epstein class. On the other hand, 
$\phi_{-1}=(E_{-1})^\ell$, where 
$E_{-1}=\alpha\circ E\circ \tau^{-1}$. Hence,
$E_{-1}: (\tau^2 x_\ell, \tau x_\ell)\to (\alpha^3, 0)$ is a diffeomorphism,
from the Epstein class, too.

Let us estimate from above the
$|E'_{-1}(1)|$.  Consider the infinitesimal cross-ratio formed
by points $\tau^2 x_\ell, 0, 1, 1+dx$ .
Note that $E_{-1}(\tau^2)=\alpha^3$, $E_{-1}(0)=\alpha$ 
and, from Lemma~\ref{covcontin}(c),
$E_{-1}(1)=1/\alpha$.   
Since $E_{-1}$ is in the Epstein class, the
cross-ratio inequality gives
\[ |E_{-1}'(1)| \frac{|\alpha^3 - \alpha|}{|\tau^2 x_\ell|}
\frac{1}{|1/\alpha-\alpha|} \frac{|1-\tau^2 x_\ell|}{|1/\alpha-\alpha^3|}  
\le 1\; \]
Note that $\frac{|1-\tau^2 x_\ell|}{|\tau^2 x_\ell|} > 1$.
Let us calculate $E_{-1}'(1)$. We use 
that $\phi_{-1}'=\ell E_{-1}^{\ell-1} E_{-1}'$ and 
$E_{-1}(1)^{\ell-1}=(1/\alpha)^{\ell-1}=|\alpha|/|\tau|$.
On the other hand, $\phi=\tau\circ \phi_{-1}\circ \phi\circ \tau^{-1}$
and so $\phi'(0)=\phi_{-1}'(1)\phi'(0)$. Thus $\phi_{-1}'(1)=1$.
We get 

\begin{equation}\label{covupper} 
1 \le \ell \frac{|\alpha|}{|\tau|}\frac{|1/\alpha-\alpha||1/\alpha-\alpha^3|}
{|\alpha^3 - \alpha|}=\frac{|1-|\tau|^{4/\ell}|}{|\tau|^{1+2/\ell}}\; .
\end{equation}
We can use the inequality $1-x<\log(1/x)$, which holds for $0<x<1$,
to obtain from~(\ref{covupper})

\[ |\tau|\le \ell |\tau|^{2/\ell}\log(|\tau|^{4/\ell})\;.\]

Now it is easy to conclude that there is $T_2$ independent on $\ell\ge 3$,
such that $|\tau|<T_2$.
\end{proof}


\subsection{Limit maps}
We follow general lines of \cite{feig} 
although some modifications and changes are necessary.
Our aim is to pick a convergent subsequence from 
$\phi_{\ell_m}$ by some kind of compactness argument. 
As $\ell_m\to \infty$,
then the domains of definition  
tend to degenerate at a limit of the critical points $x_{\ell_m}$.
To deal with this phenomenon, we consider inverse branches of $\phi_{\ell_m}$
corresponding to values to the right of
the point $x_{\ell_m}$.

By Lemma~\ref{covcontin}, each $\phi_\ell$
can be represented as $(E_{\ell}(x))^{\ell}$ with
$E_{\ell}$ an Epstein diffeomorphism with the range at least onto the
interval $(\sqrt[\ell]{A_{\ell}}, \sqrt[\ell]{\tau_{\ell}^2})$. 
Further from the same
Lemma, one gets that $E_{\ell}^{-1}
({\cal D}(\sqrt[\ell]{A_{\ell}}, \sqrt[\ell]{\tau_{\ell}^2})) 
\subset {\cal D}(B_\ell/\tau_\ell, \tau_\ell x_\ell)$. 
In the light of
Proposition~\ref{covrb} and Lemma~\ref{covcontin} (c), by taking a subsequence
we may assume that $\tau_{\ell_m} \rightarrow
\tau > 1$ and $A_{\ell_m}\to A$, $B_{\ell_m}\to B$.
Choosing yet another subsequence, we may assume that
$x_{\ell_m} \rightarrow x_0$.  Note that $1/|\tau|\le |A|\le |\tau|$,
$1\le \tau x_0\le B\le |\tau|^2$.

We will actually invert not $\phi_{\ell_m}$, but its lifting
to the universal cover, which is defined as follows. 
Consider the universal cover of the punctured disk
${\cal D}^*(\sqrt[\ell]{A_{\ell}}, \sqrt[\ell]{\tau_{\ell}^2})$
by $\exp$ and apply to the cover the linear map $w\mapsto \ell_m w$.
Then we get a $2\pi \ell_m$-periodic domain $\Pi_m$, which
contains the left half-plane and is bounded by a curve
$x=\gamma_m(y)$ ($w=x+iy$), where $\gamma_m(0)=\log \tau_{\ell_m}^2$
and $\gamma_m(\pm \pi \ell_m)=\log A_{\ell_m}$.
The real branch of the lifting of $\phi_{\ell_m}$, which
maps onto a right neighborhood of $x_{\ell_m}$, has a complex extension
\begin{equation}\label{covP}
P_{\ell_m}(w) :=
E_{\ell_m}^{-1}(\exp(w/\ell_m))
\end{equation}
which is defined in $\Pi_{m}$ and maps it into 
${\cal D}(B_\ell/\tau_\ell, \tau_\ell x_\ell)$
and so $P_{\ell_m}$ are uniformly bounded. 
By Montel's theorem we can  pick a subsequence $m_k$, such that 
$P_{\ell_{m_k}}$ converges to a mapping $P$.
Its domain of definition is 
$\Pi_\infty := \{ w :\: \Re w < \log \tau^2\}$. 
Indeed, since $\exp(w/\ell_m)$ maps 
an open arc of $\gamma_m$ between the points
with $y=\pm \pi \ell_m$ 
bijectively onto the boundary of 
${\cal D}(\sqrt[\ell_m]{A_{\ell_m}}, \sqrt[\ell_m]{\tau_{\ell_m}^2})$
(strictly, speaking, without the real point $\sqrt[\ell]{A_{\ell}}$),
and since $A_{\ell_m}, \tau_{\ell_m}$ converge, then it is easy to see
that every compact subset of 
$\Pi_\infty$ belongs to $\Pi_{m}$ for almost all $m$, and vice versa.
Since the domains
vary with $m$, they should be normalized for example by precomposing
with a translation, which tends to $0$ in the limit. This implies 
uniform convergence $P_{\ell_{m_k}}\to P$
on compact subsets.        

Let us see that $P$ is non-constant.
Note that $0\in \Pi_\infty$.
As $P_{\ell_m}(0)=0$, then $P(0)=0$.
Besides, points $\log(1/\tau_{\ell_m}^2)$ converge to
the point $\log(1/\tau^2)$, which lies in the left half-plane,
i.e., in $\Pi_\infty$. Hence, 
using Lemma~\ref{covcontin}(c), 
$P(\log(1/\tau^2))=1/\tau$.
In particular, $P$ is not a constant function.

It is also clear that $P$ is univalent. This is because
for any compact subset of $\Pi_\infty$ and $\ell_m$ large enough,
$P_{\ell_m}$ is univalent on this set, which is evident from their
defining  formulas~(\ref{covP}).

Let us define $x_0^{-} : = \lim_{x \rightarrow -\infty} P(x)$.  
Since $(P)^{-1}$ in increasing to the right of 
$x_0^-$ and $x_{\ell_m}<P_{\ell_m}(x)$ for every $m$ and real
negative $x$, 
we must have $x_0\leq x_0^-$. 
Let us also note that $x_0^-<P(\log(1/\tau^2))=1/\tau$.

We will next show that, for any $b<\tau^2$, the image of the half-plane
$\{w: \Re w < \log b\}$ by the limit map $P$ is contained
in ${\cal D}((x_0, b'),\pi/2)$ where $b'=P(\log b)$.


The inclusion 
will follow once we show that, for any $w=x+iy$ with $x<\log b$ and 
for $m$ large enough (depending on $w$),
\begin{equation}\label{conf}
\exp(w/\ell_m) \in {\cal D}((0, \sqrt[\ell_m]{b}), \pi/2).
\end{equation}
The inclusion then follows from the definition of $P$ and since
$E_{\ell_m}$ is in the Epstein class. In turn, (\ref{conf})
can be checked directly by showing that
\begin{equation}\label{covep}
\frac{|\exp(\frac{w}{\ell})-\frac{1}{2}\exp(\frac{\log b}{\ell})|}
{\frac{1}{2}\exp(\frac{\log b}{\ell})}=
|2 \exp\frac{w-\log b}{\ell} -1 |<1,
\end{equation}
if $\Re w<\log b$ and $\ell$ is large enough.

We can now define a limit mapping $\phi$ which will be shown to
belong to the ${\cal EWF}$ class.

Fix any $\tau x_0<R<\tau^2$ and
define $\Pi_*=\{w: \Re w<\log R\}$. Consider $P$ on $\Pi_*$.
We set $U = P (\Pi_*)$. Then $\phi_{|U} :=\exp\circ (P)^{-1}$.
The intersection of $U$ and the real axis
is an interval $(x_0^-, R')$, where $R'=P(\log R)$. 


We have shown that
$\phi_{\ell_m}$ converge to $\phi$ uniformly on any compact subset of 
$(x_0^-, R']$. 

As $[\log(1/\tau^2), 0]\subset \Pi_*$
and $P(0)=0$, $P(\log(1/\tau^2))=1/\tau$, there is a complex neighborhood
$S_1$ of the interval $[1/\tau, 0]$, such that $\phi_{\ell_m}\to \phi$
uniformly in $S_1$. In particular, 
$x_0\le x_0^-<1/\tau<0$ and $\phi$ is analytic in $S_1$.

Recall that $G_{\ell_m}(z)=\tau_{m}^{-1}\phi_{\ell_m}(\tau_{m}^{-1}z)$.
We define $G(z)=\lim_{m\to \infty} G_{\ell_m}(z)$ wherever the limit exists.
$G$ is defined and analytic in a complex neighborhood $S_2$ of $0$,
and $\phi=\tau^{2}\circ \phi\circ \tau^{-1}\phi\circ \tau^{-1}$ in $S_2$.
Indeed, for a small enough disk $S_2$ centered at $0$,
$\tau^{-1}S_2\subset S_2\subset S_1$, so that 
$G=\tau^{-1}\circ \phi\circ \tau^{-1}$ is well-defined and analytic in $S_2$,
and $G(S_2)$ is a small neighborhood of $G(0)=1/\tau$. As $1/\tau\in S_1$,
we may pass to the limit uniformly in $S_2$ in the equations  
$\phi_{\ell_m} = \tau_{\ell_m}^2\circ \phi_{\ell_m} \circ G_{\ell_m}$.

Now we can extend $\phi$ to an analytic map which is defined
in a complex neighborhood $S_3$ of the interval $[0,1]$ as follows.
Since $\tau^{-1}[0,1]=[1/\tau, 0]\subset S_1$, then
$G=\lim G_{\ell_m}$ is also defined and analytic
in a complex neighborhood $S_3$ of $[0,1]$, and $G(S_3)$
is a neighborhood of $\tau^{-1}\phi([1/\tau, 0])=[1/\tau, 1/\tau^3]$,
where the last interval is contained in $S_1$. Then we define 
in $S_3$: $\phi=\tau^{2}\circ \phi\circ G$. Since $S_2\subset S_3$,
then we get an analytic continuation of $\phi$. It is also clear that
$\phi_{\ell_m}\to \phi$ uniformly in a neighborhood of $[1/\tau, 1]$, and
$\phi$ is strictly increasing in $[x_0^-, 1]$.


As the next step, since $\tau^{-1}[x_0,0]\subset [0,1]$, 
$G=\lim G_{\ell_m}$, uniformly in a neighborhood $S_4$ of $[x_0,0]$.
In particular, $G(x_0)=x_0$, and $x_0$ is topologically
non-repelling: $|G(x) - x_0| \leq |x-x_0|$ for every $x\in [x_0,0]$.

\begin{lem}\label{covGfix}
$G([x_0, x_0^-]) = [x_0, x_0^-]$. 
\end{lem}
Indeed, otherwise $G(x_0^-)<x_0^-$,
and there is $x'\in (x_0^-, 0)$ such that $G(x')=x_0^-$.
Then $\phi(x')=\tau^2\circ \phi\circ G(x')=0$, 
where $x'>x_0^-$, a contradiction. This proves the lemma.

$G_{\ell_m}$ converge to $G$ uniformly on a complex
neighborhood of $[x_0,0]$. 
We use now Lemma~\ref{covcontin} (d).
Since $G'_{\ell_m}(x_{\ell_m}) = \sqrt[\ell_m]{\tau^{-2}_{m}}$, the
convergence implies $G'(x_0) = 1$. Coupled with the information
that $x_0$ is topologically non-repelling on both sides, this implies
the power-series expansion: 
\[ G(z) - x_0=(z-x_0)+a(z-x_0)^{q+1}+O(|z-x_0|^{q+1})\]
with some $a\leq 0$ and some $q$ even. 
First, we prove that $a\not=0$, i.e. $G$ is not the identity.
If $G(z)=z$, then, for every $x$ near $0$,
$\phi(x)=\phi(G(x))=\phi(x)/\tau^2$, i.e. $\phi(x)=0$, a contradiction. 
Thus, $a<0$.

Now we prove that $q=2$ considering a perturbation.
There is a fixed complex neighborhood
$W$ of $x_0$,
such that the sequence of maps $(G_{\ell_m})^{-1}$
are well-defined in $W$ and converges 
uniformly in $W$ to $G^{-1}$.
Since each $\phi_{\ell_m}$ belongs to the Epstein class,
then each $(G_{\ell_m})^{-1}$
extends to a univalent map
of the upper (and lower) half-plane
into itself. It extends also continuously on the real line,
and has there exactly $3$ fixed point $\Gamma_{\ell_m}(x_{\ell_m})$, 
$\tau^2_m X_{\ell_m}$, $x_{\ell_m}$, where $\tau^2_m X_{\ell_m}$
is strictly repelling. Therefore, by the Wolff-Denjoy theorem,
every point in either half-plane is attracted to $\tau^2_m X_{\ell_m}$
by the iterates of $(G_{\ell_m})^{-1}$, in particular, $G_{\ell_m}$
has no non-real fixed points.
This implies $q=2$ by Rouche's
principle.

Finally, we show that
$x_0=x_0^-$. Indeed, otherwise $x_0^-$ would be a fixed point of
$G$ to the right of $x_0$. Then, for big $m$, $G_{\ell_m}$ would have
either a fixed point in the upper half-plane or a fourth real fixed point,
a contradiction.

\section{Critical circle homeomorphisms}
The goal of this Section is to prove the following theorem.
\begin{theo}\label{theo:13xp,2}
Take a sequence of odd integers $\ell_n$ tending to $\infty$. 
Consider a sequence of fixed point maps $H^1_{\ell_n}$ and scaling constants $\tilde{\tau}_{\ell_n}$ introduced in Fact~\ref{fa:13xp,2} and let 
$\phi_{\ell_n}$ denote their 
branches whose domains contain $0$.  
Let $x_n$ denote the critical point of the $\phi_{\ell_n}$. 
For a subsequence $n_k$ the following are true:
\begin{itemize}
\item
sequences $x_{n_k}$ and $\tilde{\tau}_{\ell_{n_k}}$ converge to $x_0$ and $\tau$, respectively, which satisfy the inequalities postulated by Definition~\ref{defi:9xa,1}, 
\item
mappings $\phi_{\ell_{n_k}}$ converge almost uniformly on 
$(x_0,x_0/\tau)$ to $\phi$ which belongs to the ${\cal EWF}$-class.
\end{itemize}
\end{theo}

\subsection{Fixed-point equations}
\paragraph{Map near critical point}
Fix an odd integer $\ell\ge 3$, and let $H^1_\ell$ be the map from the 
Fact~\ref{fa:13xp,2}. Consider the corresponding map $\tilde h_\ell$ near
the critical point, i.e., $\tilde h_\ell=p^{-1}\circ H^1_\ell\circ p$,
where $p(x)=x^\ell$. 
It consists of two branches
$\tilde g_i=p^{-1}\circ \phi_i\circ p$, where $\phi_i$, $i=0,-1$ are the
branches of $H^1_\ell$, so that $\tilde g_i$ is defined on $J^i=p^{-1}(I^i)$
with the common image $J=p^{-1}(I)$. The intervals $J_0, J^{-1}$
have a common endpoint.
The scaling factor for $\tilde h_\ell$ is
$\tilde\alpha_\ell=\tilde \tau_\ell^{1/\ell}<-1$.
The functional equations for $H^1_\ell$ imply that 
$\tilde g_i$ satisfy similar equations:
\begin{equation}\label{homeofixp0}
\tilde g_{-1}=\tilde\alpha_\ell\circ \tilde g_{0}\circ \tilde
\alpha_\ell^{-1}(x), 
\end{equation}
\begin{equation}\label{homeofixp00}
\tilde\alpha_\ell^{-1}\circ \tilde g_{0}\circ \tilde\alpha=
\tilde g_{-1}\circ \tilde g_{0}. 
\end{equation}
Furthermore,
by~\cite{defaria},~\cite{demelofaria}, $\tilde g_0$ extends 
in a real-analytic fashion through
the left end point of $J^0$ to a neighborhood of the interval 
$[\tilde\alpha_\ell, 0]$, and similarly $\tilde g_{-1}$
extends to a real-analytic homeomorphism defined in a neighborhood 
of the interval $[0,1]$, so that 
\begin{equation}\label{comm}
\tilde g_0\circ \tilde g_{-1}=\tilde g_{-1}\circ \tilde g_{0}
\end{equation}
near the point $0$.
Finally, $\tilde g_0(x)=E(x^\ell)$ where a diffeomorphism $E=E_\ell$
from a neighborhood of $[\tilde \alpha^\ell, 0]$ onto its image
belongs to the Epstein class.

\begin{com}\label{homeopair}
The pair of maps $f_-=\tilde g_0: [1/\tilde\alpha_\ell, 0]
\to [\tilde g_0(1/\tilde\alpha_\ell), 1]$,
$f_+=\tilde\alpha_\ell^{-1}\circ f_-\circ 
\tilde\alpha_\ell: [0,1]\to [f_+(0),  f_+(1)]$
is a commuting pair in the sense of ~\cite{defaria},~\cite{demelofaria},
which is the unique fixed point of the renormalization operator
corresponding to the golden mean rotation number.
To be more precise, the equations~(\ref{homeofixp0})-
(\ref{homeofixp00}),~(\ref{comm}) mean that
$$\tilde\alpha_\ell^{-1}\circ \tilde g_{0}\circ \tilde\alpha_\ell=
\tilde g_{-1}\circ \tilde g_{0}=\tilde g_{0}\circ \tilde g_{-1},$$
and they
are equivalent to the following two conditions:
$$f_+=\alpha\circ f_+\circ f_-\circ \alpha^{-1}, \ \ 
f_+\circ f_-=f_-\circ f_+.$$
\end{com}

We introduce a point $y_\ell$, which is the unique zero of $\tilde g_0$
in $[\alpha, 0]$. Define also an associated dynamics
$\gamma=\tilde\alpha_\ell\circ g_0\circ \tilde\alpha_\ell^{-2}$.
Note that the map $\gamma$ depends on $\ell$.

\begin{lem}\label{homeocontin0}
(a) $\tilde g_0$ extends to a real-analytic orientation preserving
homeomorphism from $(\tilde\alpha_\ell^2 y_\ell, \tilde\alpha_\ell y_\ell)$ 
onto
$(\tilde\alpha_\ell, \tilde\alpha_\ell^2)$, 
where it has a representation $\tilde g_0(x)=E(x^\ell)$
with $E$ a diffeomorphism in the Epstein class.
In particular, $\tilde\alpha_\ell^2 y_\ell< \tilde\alpha_\ell$.

(b) $\gamma: (\tilde\alpha_\ell^2 y_\ell, 0)\to (\tilde\alpha_\ell, 0)$
is an orientation reversing diffeomorphism, which has a unique
fixed point $y_\ell$. Moreover, $\gamma'(y_\ell)=1/\tilde\alpha_\ell
\in (-1, 0)$.
We have: $\tilde\alpha_\ell^{-1}\circ \tilde g_0=\tilde g_0\circ \gamma$ 
wherever both sides are well-defined.

(c) $\gamma^2=\tilde\alpha_\ell^{-1}\circ 
\tilde g_0\circ \tilde\alpha_\ell^{-1}:
(\tilde\alpha_\ell^2 y_\ell, 0)\to 
(\tilde\alpha_\ell, 1/\tilde\alpha_\ell)$ is an orientation preserving
diffeomorphism, which has a unique fixed point at $y_\ell$.
\end{lem}
\begin{proof}
We have formally: $\tilde g_0\circ \gamma=\tilde\alpha_\ell^{-1}
\circ \tilde g_0$.
On the other hand, both sides are well-defined neat $y_\ell$,
and $\gamma(y_\ell)\in (\tilde\alpha_\ell, 0)$. Hence,
$\tilde g_0(\gamma(y_\ell))=0$ implies that $\gamma(y_\ell)=y_\ell$.
In particular, $\tilde g_0'(y_\ell)\gamma'(y_\ell)=
\tilde\alpha_\ell^{-1} \tilde g_0'(y_\ell)$,
hence, $\gamma'(y_\ell)=\tilde\alpha_\ell^{-1}$.
We have: $\tilde\alpha_\ell\circ \tilde g_0\circ 
\tilde\alpha_\ell^{-1}\circ \tilde g_0=
\tilde\alpha_\ell^{-1}\circ \tilde g_0\circ 
\tilde\alpha_\ell$. Applying this for $x=0$ we get 
$\tilde g_0(1/\tilde\alpha_\ell)=
1/\tilde\alpha_\ell^2$. Since $0>y_\ell>\tilde\alpha_\ell$, we can write:
$y_\ell/\tilde\alpha_\ell=
\tilde g_0(y_\ell/\tilde\alpha_\ell^2)>\tilde g_0(1/\tilde\alpha_\ell)=
1/\tilde\alpha_\ell^2$,
i.e., $\tilde\alpha_\ell y_\ell>1$ and $\tilde\alpha_\ell^2 y_\ell
<\tilde\alpha_\ell$.
But $\gamma$ is a diffeomorphism  of a neighborhood of
$[\tilde\alpha_\ell^2 y_\ell, 0)$ onto a neighborhood of 
$(\tilde\alpha_\ell, 0]$.
Hence, the formula 
$\tilde g_0=\tilde\alpha_\ell\circ \tilde g_0\circ \gamma$ gives us an 
analytic continuation of $\tilde g_0$ to a neighborhood of
$[\tilde\alpha_\ell^2 y_\ell, 0)$,
with $\tilde\alpha_\ell^2 y_\ell$ the only critical point and with
$\tilde g_0(\tilde\alpha_\ell^2 y_\ell)=\tilde\alpha_\ell$.

Formally, $\gamma^2=\tilde\alpha_\ell^{-1}\circ \gamma\circ 
\tilde\alpha_\ell^{-1}$.
The latter map is an orientation preserving
diffeomorphism of a neighborhood of
$(0, \tilde\alpha_\ell y_\ell]$ onto a neighborhood
of $(1/\tilde\alpha_\ell, 0]$. Then the formula
$\tilde g_0=\tilde\alpha_\ell^2\circ \tilde g_0\circ \gamma^2$ 
defines an analytic continuation
of $\tilde g_0$ to $(0, \tilde\alpha_\ell y_\ell]$ 
with $\tilde\alpha_\ell y_\ell$ the only
critical point and with 
$\tilde g_0(\tilde\alpha_\ell y_\ell)=\tilde\alpha_\ell^2 
\tilde g_0(0)=\tilde\alpha_\ell^2$.
The rest follows easily.  
\end{proof}

Let us come back to the  $H^1_\ell$.
It consists of the pair
$\phi_j=p\circ \tilde g_j\circ p^{-1}$, $j=0,-1$, 
which are defined on the intervals 
$(x_\ell, x_\ell/\tilde\tau_\ell)$ and 
$(x_\ell/\tilde\tau_\ell, \tilde\tau_\ell x_\ell)$
respectively, where $x_\ell=y_\ell^\ell$. 
Notice that $\phi_0(x)=(E(x))^\ell$.
Recall that 
\begin{equation}\label{homeofixp}
\phi_{-1}=\tau\circ \phi_{0}\circ \tau^{-1}, \ \ \ \ \ \ \ \ \ \ \
\phi_{0}=\tau\circ \phi_{-1}\circ \phi_{0}\circ \tau^{-1}.
\end{equation}  
We denote 
$$\tilde\Gamma_\ell=p_\ell\circ \gamma\circ p_\ell^{-1}, \ \ \ 
\tilde G_\ell=\tilde\Gamma_\ell^2.$$



Then Lemma~\ref{homeocontin0} is reformulated as follows
(except for the part (d) below, which still needs to be proved).
\begin{lem}\label{homeocontin}
(a) $\phi=\phi_0$ extends to a real-analytic orientation preserving
homeomorphism from $(\tilde\tau_\ell^2 x_\ell, 
\tilde\tau_\ell x_\ell)$ onto
$(\tilde\tau_\ell, \tilde\tau_\ell^2)$, 
where it has a representation $\phi(x)=(E(x))^\ell$
with $E=E_\ell$ a diffeomorphism in the Epstein class, $E(0)=1$.

(b) $\tilde\Gamma_\ell: (\tau^2 x_\ell, 0)\to (\tau, 0)$
is an orientation reversing diffeomorphism, which has a unique
fixed point at $x_\ell$. Moreover, $\tilde\Gamma_\ell'(x_\ell)=
1/\tilde\alpha_\ell\in (-1, 0)$.
We have: $\tilde\tau_\ell^{-1}\circ \phi=\phi\circ \tilde\Gamma_\ell$ 
wherever both sides are well-defined.

(c) $\tilde G_\ell=\tilde\tau_\ell^{-1}\circ \phi\circ \tilde\tau_\ell^{-1}:
(\tilde\tau_\ell^2 x_\ell, 0)\to (\tilde\tau_\ell, 1/\tilde\tau_\ell)$ 
is an orientation preserving
diffeomorphism, which has a unique fixed point at $x_\ell$.

(d) $\phi(1/\tilde\tau_\ell)=1/\tilde\tau_\ell^2$, 
$\phi'(1/\tilde\tau_\ell)=1$,
$\phi(x_\ell/\tilde\tau_\ell)=\tilde\tau_\ell x_\ell$.
Also,
$1<\tilde\tau_\ell x_\ell<\tilde\tau_\ell^2$.
\end{lem}

It remains to check (d), and it is done very similar to the proof
of the part (c) of Lemma~\ref{covcontin} as the equations are identical.
First,
$\tilde\tau_\ell^{-2}=\tilde\tau_\ell^{-2}\phi(0)=
\phi\circ \tilde\tau_\ell^{-1}\phi\circ \tilde\tau_\ell^{-1}(0)=
\phi(1/\tilde\tau_\ell)$. Besides, 
$\tilde\tau_\ell^{-2}\phi'(0)=\tilde\tau_\ell^{-2}\phi'(1/\tilde\tau_\ell)
\phi'(0)$,
i.e., $\phi'(1/\tilde\tau_\ell)=1$. Finally,
$\tilde\tau_\ell^2=\phi(\tilde\tau_\ell x_\ell)>\tilde\tau_\ell x_\ell=
\phi(x_\ell/\tilde\tau_\ell)>1=\phi(0)$.

\subsection{Bounds} 
\paragraph{Real bounds.}

\begin{prop}\label{homeorb}
There exist
two constants $1<T_1<T_2<\infty$,
such that $T_1<|\tilde\tau_\ell|<T_2$, for all $H^1_\ell$.
\end{prop}

The proof is contained in the following two lemmas~\ref{homeoT_1},
~\ref{homeoT2}.

\begin{lem}\label{homeoT_1}
There exists $1<T_1$,
such that $T_1<|\tilde\tau_\ell|$ for all $H^1_{\ell}$.
\end{lem}
\begin{proof} 
We use an idea from~\cite{gregbra}
Let $f$ be a critical circle homeomorphism with the golden mean
rotation number,
a single
critical value at $0=f(c)$, where $c$ is the critical point of an integer 
odd order $\ell$,
and $f$ is $C^3$ with negative Schwarzian. Then, by~\cite{defaria},
\cite{demelofaria}
$\phi$ on $[1/\tilde\tau_\ell, 0]$ is the uniform limit of a sequence of maps
$A_n\circ f^{q_n}\circ A_n^{-1}$, where $q_n$ is the Fibonacci sequence
and $A_n$ is a linear map (from the angle coordinate of the unit circle
into reals), which maps $f^{q_{n}}(0)$ to $1$.
We will consider also $f_t=f-t$, for small real $t>0$, and denote,
for $N$ integer and $0\le t'\le t$, $N(t')=f_{t'}^{N}(0)$.
Note that when $t$ moves back to $0$, all positive iterates of any
point $x$ move to the right, and all negative move to the left.
Let us fix a minimal positive $t$, such that
$f_t$ has a periodic orbit of period $q_{n+2}$. It implies
the following fact (*), to which we will often refer:

(*) for every point $x$, the ordering on the circle of points 
$f_{t'}^m(x)$, for $0\le m\le q_{m+2}-1$, does not depend on $t'\in [0, t]$.

We denote by $|(a,b)|$ the length of an interval with the end points
$a$ and $b$.

{\bf Claim.} {\it There is $K>0$ such that, for every $\ell$
and every $n$ large enough,
either
\begin{equation}\label{forw}
\frac{|(q_n(t), 2q_n(t))|}{|(0, q_n(t))|}\ge K,
\end{equation}
or
\begin{equation}\label{back}
\frac{|(-2q_n(t), -q_n(t))|}{|(-q_n(t), 0)|}\ge K,
\end{equation}
}
Let us see the Claim implies the Lemma.
Assume (\ref{forw}) holds, for a fixed $\ell$
and every $n$ big enough. Note that, 
$0<-q_{n+1}(0)<q_n(t)<2q_n(t)<2q_n(0)<-q_{n-3}(0)$. 
On the other hand, as $n\to \infty$, $-q_{n-3}(0)/-q_{n+1}(0)\to 
\tilde\tau_\ell^4$.
Hence, for $n$ big, 
$$\tilde\tau_\ell^4>
\frac{|(-q_{n+1}(0), 2q_n(0))|}{|(0, -q_{n+1}(0))|}
>\frac{|(q_n(t), 2q_n(t))|}{|(0, q_n(t))|}\ge K.$$
Similarly, let (\ref{back}) hold. Then
$0>q_{n+1}(0)>-q_n(t)>-2q_n(t)>-2q_n(0)>q_{n-3}(0)$
while, as $n\to \infty$, $q_{n-3}(0)/q_{n+1}(0)\to \tilde\tau_\ell^4$.
Hence, for $n$ big, 
$$\tilde\tau_\ell^4>
\frac{|(-2q_n(0), q_{n+1}(0))|}{|(q_{n+1}(0), 0)|}
>\frac{|(-2q_n(t), -q_n(t))|}{|(-q_n(t), 0)|}\ge K.$$
Let us prove the Claim. Consider a partition of the unit circle
by the points, $N(t)$, for $0\le N\le q_{n+2}-1$. Note that
$q_{n+2}(t)=0$. Denote by $J$ the shortest interval of this partition.

(1) $J_0=(0, q_n(t))$ is an interval of the partition. 
Indeed, if it contains some $i(t)$, $1\le i\le q_{n+2}-1$,
then, by (*), the same holds with $t=0$, a contradiction, because
$q_{n+2}(0)$ is the first return of $0$ to $(0, q_n(0))$.

Consider the configuration of $4$ points $-q_n(t), 0, q_n(t), 2q_n(t)$.
First, we apply to it $f_t^{q_{n+1}}$.

(2) Let us check that $F_t=f_t^{q_{n+1}}$ is a diffeomorphism
on $(-q_n(t), 2q_n(t))$ except for the last iterate when 
$(q_n+q_{n+1}-1)(t)$ is the critical point of $f_t$.

(i) $F_t(q_n(t), 2q_n(t))=(0, q_n(t))$, hence, by (1),
$F_t$ is a diffeomorphism on $(q_n(t), 2q_n(t))$.

(ii) $F_t(J_0)=(q_{n+1}(t), 0)$. The latter interval contains
no $i(t)$, $1\le i\le q_{n+1}-1$, because this is true for $t=0$.
Thus $F_t$ is a diffeomorphism on $J_0$.

(iii) $F_t=f_t^{q_{n-1}}\circ f_t^{q_n}$. Since 
$f_t^{q_n}((-q_n(t), 0))=J_0$, $f_t^{q_n}$ is a diffeomorphism
on $(-q_n(t), 0)$. Also, $f_t^{q_{n-1}}$ is a diffeomorphism
on $J_0$. Indeed, $f_t^{q_{n-1}}(J_0)=(q_{n-1}(t), q_{n+1}(t))$,
and the latter interval contains no $i(t)$, $0\le i\le q_{n-1}-1$,
because the same is true for $t=0$ (we again use (*)).

Thus (2) checked. 
By this, if $f_t^i(J_0)=J$, for some
$0\le i\le q_{n+1}-1$, then we employ the shortest interval argument and
arrive at (\ref{forw}). If $f_t^{q_{n+1}}(J_0)=(-q_n(t), 0)$ is $J$,
then we get immediately (\ref{back}).

(3) Otherwise, we must meet $J$ when applying $f_t^i$
(for some
$1\le i\le q_n-1$) to the interval $(-q_n(t), 0)$ inside of the configuration
$-2q_n(t), -q_n(t), 0, q_n(t)$. We have checked in (2) that $f_t^{q_n-1}$
is a diffeomorphism on $[-q_n(t), q_n(t)]$. But $f_t^{q_n-1}$
is a diffeomorphism on $(-2q_n(t), -q_n(t))$, too, because
otherwise $(-q_n(t), 0)$ contains some $i(t)$, $1\le i\le q_n-1$, hence,
$J_0$ contains $(i+q_n)(t)$. Since $i+q_n<q_{n+2}$, this is
impossible by the step (1). Thus $f_t^j((-q_n(t), 0))=J$, for
some $1\le i\le q_n-1$ where $f_t^i$ is a diffeomorphism
on $(-2q_n(t), q_n(t))$. Then, by the shortest interval argument,
(\ref{back}) follows.

\end{proof}

\begin{lem}\label{homeoT2}
There exists 
$T_2$ such that for all $H^1_{\ell}$, 
we get $|\tilde\tau_{\ell}| < T_2$. 
\end{lem}
\begin{proof} 
Consider the map $\phi_{-1}=\tau\circ \phi\circ \tau^{-1}$.
By Lemma~\ref{homeocontin} (a), 
\[ \phi_{-1}: (\tilde\tau_\ell^2 x_\ell, \tilde\tau_\ell x_\ell)
\to (\tilde\tau_\ell^3, 0)\] is a diffeomorphism
from the Epstein class. 
Then we can proceed word by word as in the proof of Lemma~\ref{covT2}.
\end{proof}

\subsection{Limit maps}
Consider inverse branches of $\phi_{\ell_m}$
corresponding to values to the right of
the point $x_{\ell_m}$.
By Lemma~\ref{homeocontin}, each $\phi_\ell$
can be represented as $(E_{\ell}(x))^{\ell}$ with
$E_{\ell}: (\tilde\tau_\ell^2 x_\ell, \tilde\tau_\ell x_\ell)\to
(\sqrt[\ell]{\tilde\tau_\ell}, \sqrt[\ell]{\tilde\tau_\ell^2})$ 
an Epstein diffeomorphism.
Hence, 
\[ E_{\ell}^{-1}
(({\cal D}(\sqrt[\ell]{\tilde\tau_\ell}, \sqrt[\ell]{\tilde\tau_\ell^2}))
\subset {\cal D}(\tilde\tau_\ell^2 x_\ell, \tilde\tau_\ell x_\ell)\; . \]
In the light of
Proposition~\ref{homeorb} and Lemma~\ref{homeocontin}, by taking a subsequence
we may assume that $\tilde\tau_{\ell_m} \rightarrow
\tau > 1$ and 
$x_{\ell_m} \rightarrow x_0$.  Note that 
$1\le \tau x_0\le |\tau|^2$.

Consider the universal cover of the punctured disk
${\cal D}^*(\sqrt[\ell]{\tilde\tau_{\ell}}, \sqrt[\ell]{\tilde\tau_{\ell}^2})$
by $\exp$ and apply to the cover the linear map $w\mapsto \ell_m w$.
Then we get a $2\pi \ell_m$-periodic domain $\Pi_m$, which
contains the left half-plane and is bounded by a curve
$x=\gamma_m(y)$ ($w=x+iy$), where $\gamma_m(0)=\log \tilde\tau_{\ell_m}^2$
and $\gamma_m(\pm \pi \ell_m)=\log \tilde\tau_{\ell_m}$.
The real branch of the lifting of $\phi_{\ell_m}$, which
maps onto a right neighborhood of $x_{\ell_m}$, has a complex extension
\begin{equation}
P_{\ell_m}(w) :=
E_{\ell_m}^{-1}(\exp(w/\ell_m))
\end{equation}
which is defined in $\Pi_{m}$ and maps it into 
${\cal D}(\tilde\tau_\ell^2 x_\ell, \tilde\tau_\ell x_\ell)$
and so $P_{\ell_m}$ are uniformly bounded. 
Pick a subsequence $m_k$, such that 
$P^{\ell_{m_k}}$ converge to a mapping $P$.
Its domain of definition is 
$\Pi_\infty := \{ w :\: \Re w < \log \tau^2\}$. 
This implies 
uniform convergence $P_{\ell_{m_k}}\to P$
on compact subsets.        

Let us see that $P$ is non-constant.
Note that $0\in \Pi_\infty$.
As $P_{\ell_m}(0)=0$, then $P(0)=0$.
Besides, points $\log(1/\tilde\tau_{\ell_m}^2)$ converge to
the point $\log(1/\tilde\tau^2)$, which lies in the left half-plane,
i.e., in $\Pi_\infty$. Hence, 
using Lemma~\ref{homeocontin}(c), 
$P(\log(1/\tilde\tau^2))=1/\tilde\tau$.
In particular, $P$ is not constant.
$P$ is univalent because
for any compact subset of $\Pi_\infty$ and $\ell_m$ large enough,
$P_{\ell_m}$ is univalent on this set, which is evident from 
its definition.
Let us define $x_0^{-} : = \lim_{x \rightarrow -\infty} P(x)$.  
Since $(P)^{-1}$ in increasing to the right of 
$x_0^-$ and $x_{\ell_m}<P_{\ell_m}(x)$ for every $m$ and real
negative $x$, 
we must have $x_0\leq x_0^-$. 
Let us also note that $x_0^-<P(\log(1/\tau^2))=1/\tau$.

Precisely as in the covering maps case, we show that, for any $b<\tau^2$,
the image of the half-plane
$\{w: \Re w < \log b\}$ by the limit map $P$ is contained
in ${\cal D}((x_0, b'),\pi/2)$ where $b'=P(\log b)$.

We can now define a limit mapping $\phi$ which will be shown to
belong to the ${\cal EWF}$ class.

Fix any $\tau x_0<R<\tau^2$ and
define $\Pi_*=\{w: \Re w<\log R\}$. Consider $P$ on $\Pi_*$.
We set $U = P (\Pi_*)$. Then $\phi_{|U} :=\exp\circ (P)^{-1}$.
The intersection of $U$ and the real axis
is an interval $(x_0^-, R')$, where $R'=P(\log R)$. 


We have shown that
$\phi_{\ell_m}$ converge to $\phi$ uniformly on any compact subset of 
$(x_0^-, R']$. 

As $[\log(1/\tau^2), 0]\subset \Pi_*$
and $P(0)=0$, $P(\log(1/\tau^2))=1/\tau$, there is a complex neighborhood
$S_1$ of the interval $[1/\tau, 0]$, such that $\phi_{\ell_m}\to \phi$
uniformly in $S_1$. In particular, 
$x_0\le x_0^-<1/\tau<0$ and $\phi$ is analytic in $S_1$.

Recall that $\tilde G_{\ell_m}(z)=\tau_{m}^{-1}\phi_{\ell_m}(\tau_{m}^{-1}z)$.
We define $\tilde G(z)=\lim_{m\to \infty} \tilde G_{\ell_m}(z)$ 
wherever the limit exists.
 
Repeating the proof for the covering maps, 
we extend $\phi$ to an analytic map which is defined
in a complex neighborhood $S_3$ of the interval $[0,1]$,
and 
$\tilde G$ is defined and analytic in a complex neighborhood $S_4$ of 
$[x_0, 0]$.
In particular, $\tilde G(x_0)=x_0$, and $x_0$ is topologically
non-repelling: $|\tilde G(x) - x_0| \leq |x-x_0|$ for every $x\in [x_0,0]$.
Then we have:
\begin{lem}
$\tilde G([x_0, x_0^-]) = [x_0, x_0^-]$. 
\end{lem}
  
$\tilde G_{\ell_m}$ converge to $\tilde G$ uniformly on a complex
neighborhood of $[x_0,0]$. 
We use now Lemma~\ref{homeocontin} (d).
Since $\tilde G'_{\ell_m}(x_{\ell_m}) = \sqrt[\ell_m]{\tilde\tau^{-2}_{m}}$, 
the
convergence implies $\tilde G'(x_0) = 1$. Coupled with the information
that $x_0$ is topologically non-repelling on both sides, this implies
the power-series expansion: 
\[ G(z) - x_0=(z-x_0)+a(z-x_0)^{q+1}+O(|z-x_0|^{q+1})\]
with some $a\leq 0$ and some $q$ even. 
First, we prove that $a\not=0$, i.e. $G$ is not the identity.
If $G(z)=z$, then, for every $x$ near $0$,
$\phi(x)=\phi(G(x))=\phi(x)/\tau^2$, i.e. $\phi(x)=0$, a contradiction. 
Thus, $a<0$.

Now we prove that $q=2$ considering a perturbation.
There is a fixed complex neighborhood
$W$ of $x_0$,
such that the sequence of maps $(\tilde G_{\ell_m})^{-1}$
are well-defined in $W$ and converges 
uniformly in $W$ to $\tilde G^{-1}$.
Since each $\phi_{\ell_m}$ belongs to the Epstein class,
then each $(\tilde G_{\ell_m})^{-1}$
extends to a univalent map
of the upper (and lower) half-plane
into itself. It extends also continuously on the real line,
and has there exactly one fixed point, which is
$x_{\ell_m}$ and which is repelling. Therefore, by the Wolff-Denjoy theorem,
$(\tilde G^2)^{-1}$ has at most one fixed point in either half-plane, and one
which is strictly attracting.
This implies $q=2$ by Rouche's
principle.

Finally, we show that
$x_0=x_0^-$. Indeed, otherwise $x_0^-$ would be a fixed point of
$\tilde G$ to the right of $x_0$. Then, for big $m$, $\tilde G_{\ell_m}$ 
would have
either two fixed points in the upper half-plane or a second real fixed point,
a contradiction.

\section{Dynamics of {\cal EWF} Maps}
\subsection{Three levels of dynamics.}
There are three types of dynamics which can be associated with a map
$\phi$ from the ${\cal EWF}$-class: dynamics of ${\cal H}$ on the circle,
its extension to the complex dynamics similar to a polynomial-like
map, and a so-called presentation function. Dynamics on the circle is
easy to construct and was already described when we defined 
${\cal EWF}$-class. 

\paragraph{Construction of the complex dynamics $H$.}
The complex dynamical system will consists of two branches mapping onto the same range 
${\cal D}(x_0,\tau x_0)\setminus \{0\}$. The first one is $\phi$. Its domain $\Omega_-$ is going to be contained in the domain 
$U$ introduced in the definition of class ${\cal EWF}$. In particular it is contained in 
${\cal D}(x_0,x_0/\tau)$ and, obviously, in ${\cal D}(x_0,\tau x_0)$. 

The second branch is $\phi_{-1}$ mapping onto the same range. Its domain $\Omega_+$ by rescaling is contained in $\tau U$ and hence in 
${\cal D}(x_0,\tau x_0)$. However, we want to emphasize that even though $\Omega_+\cap\RR = (x_0/\tau,x_0\tau)$, $\Omega_+$ has no reason to be contained in 
${\cal D}(x_0/\tau,x_0\tau)$. We have the following lemma, though: 

\begin{lem}\label{lem:28ua,2}
\[ \overline{\Omega}_- \cap \overline{\Omega}_+ = \{\frac{x_0}{\tau}\} \; .\]
\end{lem}
\begin{proof}
We will prove that 
\[ \overline{\phi^{-1}\left(\partial {\cal D}(x_0,\tau
  x_0)\setminus\{0\}\right)} \cap  
\overline{\phi^{-1}_{-1}\left(\partial {\cal D}(x_0,\tau
  x_0)\setminus\{0\}\right)} = \{ \frac{x_0}{\tau} \} \; .\]

Suppose that $z$ belongs to the intersection of the boundaries of 
\[ \overline{\phi^{-1}\left(\partial {\cal D}(x_0,\tau x_0)\setminus\{0\}\right)}\]
 and 
\[ \overline{\phi^{-1}_{-1}\left(\partial {\cal D}(x_0,\tau
  x_0)\setminus\{0\}\right)} . \] 

If $y=\phi_{-1}(z)$, then 
\[ y\in \partial{\cal D}(x_0,\tau x_0)\; .\] 
By Lemma~\ref{lem:9xa,1}, $\phi_{-2}(y) = \phi(z) \in \partial{\cal D}(x_0,\tau x_0)$. But the preimage of ${\cal D}(x_0,\tau
x_0)$ under $\phi_{-2}$ is contained 
in 
\[ {\cal D}(\tau^2x_0,\tau^2\phi^{-1}(\frac{x_0}{\tau})) = {\cal D}(\tau^2 x_0, x_0) \]
by Definition~\ref{defi:9xa,1}. Hence, the only possibility is for $y$ to be equal to $x_0$. In other words, $z$ must be 
a preimage of $x_0$ by $\phi_{-1}$. 

Again by item 3 of Definition~\ref{defi:9xa,1} such preimages belong to the closure of ${\cal D}(\frac{x_0}{\tau},\tau x_0)$ which intersects 
the closure of
\[ \overline{\phi^{-1}\left(\partial {\cal D}(x_0,\tau x_0)\setminus\{0\}\right)} \subset {\cal D}(x_0,\frac{x_0}{\tau}) \]
only at $\frac{x_0}{\tau}$. 
\end{proof}

\paragraph{Dynamics on the circle.}
The circle is formed by restricting $H$ to the interval $(x_0,\tau x_0)$ and identifying $x_0$ with $\tau x_0$. 
Since $\phi(\tau^{-1}x_0) =\tau x_0$ and $\phi_{-1}(\tau^{-1} x_0) = x_0$ while $\phi(x_0) = 0 = \phi_{-1}(\tau x_0)$ this gives a degree $1$ circle homeomorphism which is
smooth except at $\tau^{-1} x_0$. The rotation number of the circle dynamics generated by any ${\cal EWF}$-map equals the golden mean by definition.  
 
\paragraph{The presentation function.}
The presentation function $\Pi$ is defined 
on the interval $[x_0,\tau x_0]$ as follows. 
For $x\leq x_0/\tau$, $\Pi(x) = \tau x$. For $x>x_0/\tau$, 
$\Pi(x) = \phi_{-1}(x)$. Note that $\Pi$ is not continuously defined on 
the circle.

\subsection{Dynamics of the real presentation function}
The presentation function is useful in the study of the dynamics of
$H$ because it is much simpler dynamically, but contains the full
information about the post-critical set of $H$.  Unlike $H$, $\Pi$ is
post-critically finite, since $\Pi(x_0\tau) = 0$ is the repelling fixed 
point of $\Pi$. Note also that
the points
$1,\tau^{-1}$ form a periodic orbit under $\Pi$. This will be used later to construct the conjugacy 
between presentation functions, but in this section we limit ourselves to the real dynamics of presentation functions.

As usual for post-critically finite maps, we get the following:

\begin{lem}\label{lem:8xp,2}
The first return  map of $\Pi$ into $(x_0,x_0/\tau^2)$ is defined
everywhere expect for a countable set and consists of diffeomorphic
branches with negative Schwarzian derivative, all extendable to a
fixed interval which contains $[x_0,x_0\tau^2]$ in its interior.
\end{lem}
\begin{proof} 
By $\phi_{-1} \tau = \Pi^2$ interval $[x_0,x_0/\tau^2]$ is mapped onto $[x_0,0]$ and extendable to $(\tau,0)$ as a diffeomorphism. The first return map is obtained by 
composing this map piecewise with iterates of $\Pi^2 = \tau^2$. This results in a map defined except on a countable set with all branch extendable dynamically 
to $(\tau,0)$. These extended branches are all compositions of $\phi_{-1}$ and $\tau$ and thus have negative Schwarzian derivative by the setup of the ${\cal EWF}$-class. 
\end{proof}

\begin{lem}\label{lem:28ua,4}
Preimages of $1$ under $\Pi$ are dense in $[x_0,\tau x_0]$. 
\end{lem}
\begin{proof}
From Lemma~\ref{lem:8xp,2}, preimages of any point in $[x_0,x_0/\tau^2]$ by iterates of $\Pi$ are dense in that interval. 
But $[x_0,x_0/\tau^2]$ is a fundamental domain for the dynamics: every orbit under $\Pi$ passes through that interval. Hence, the preimages of any
point of $[x_0,x_0/\tau^2]$ are dense in $[x_0,\tau x_0]$. That includes preimages of point $\tau^{-1} = \Pi^{-1}(1)$. 
\end{proof} 

The connection of $\Pi$ with $H$ is summarized by this lemma. 

\begin{lem}\label{lem:8xp,1}
A point $x\in (x_0,\tau x_0)$ is equal to $H^j(0)$, $j>0$, if and only
if $\Pi^k(x)=1$ for some $k\geq 0$. 
\end{lem}
\begin{proof}
In this proof we adopt the following notation which is consistent with a later use for towers: $\phi_k = \tau^{-k} \phi \tau^k$ for $k\in\ZZ$.  
We can always write for $j>0$: 
\[ H^j(0) = \phi_{\epsilon_j} \circ \cdots \circ \phi_{\epsilon_1} (0) \]
for some $\epsilon_m =0,-1$ for $m=1,\cdots,k$. By Fibonacci combinatorics, this sequence does not contain two $-1$'s in a row.

Assume first the $H^j(0) < x_0/\tau$. Then 
\[ \Pi(H^j(0)) = \tau H^j(0) = \phi_{\epsilon_j-1}\circ\cdots\circ \phi_{\epsilon_1-1}(0) = \phi_{\epsilon'_{k}}\circ \cdots \phi_{\epsilon'_1}(0) = H^k(0) \]
where the final representation is obtained from the functional equation~(\ref{equ:27up,1}) in the form $\phi_{-2}\circ \phi_{-1} = \phi_0$. 
Combinatorially, the whole process is equivalent 
to taking the sequence $(\epsilon_m)_{m=1}^j$ and whenever $\epsilon_m=0$ and $\epsilon_{m+1}=-1$ replacing it with $0$, while
any occurrence of $\epsilon_m=0$ which is not followed by $\epsilon_{m+1}=-1$ is replaced with $-1$. Observe that $k\leq j$ and $j-k$ is 
equal to the number of occurrences of $-1$ in the sequence $(\epsilon_m)_{m=1}^j$. 
If $H^j(0)>x_0/\tau$, then $\Pi^2(H^j(0)) = \Pi(H^{j+1}(0))$ 
where $H^{j+1}(0) < x_0/\tau$. Since $\epsilon_{j+1}=-1$, by applying the previous case we see that 
$\Pi^2(H^j(0)) = H^k(0)$ where $k\leq j$ and again $j-k$ equals the number of occurrences of $-1$ in the sequence $(\epsilon_m)_{m=1}^j$.  

We conclude that the first or the second iteration of $\Pi$ maps $H^j(0)$ to $H^k(0)$ where $k<j$ unless the sequence $(\epsilon_m)$ consists only of zeros. That is only 
possible for $j=1$ and thus every $H^j(0)$ is a preimage of $H(0)=1$ under the iterates of $\Pi$. 

Conversely, $\Pi^{-1}(H^j(0))$ could either be $\tau^{-1} H^j(0) = H_1^j(0)$ 
or $\phi_{-1}^{-1} H^j(0)$. In the first case, $H_1$ has two branches, one of which is $\phi_0$ and another $\phi_1 = \phi_{-1}\circ \phi_0$ 
by the functional equation~(\ref{equ:27up,1}). Consequently, in this case $\Pi^{-1}(H^j(0)) = H^k(0)$ for $k>j$. In the second case $\Pi^{-1}(H^j(0)) = H^{j-1}(0)$. Hence, it 
follows by induction that every preimage of $1=H(0)$ by $\Pi$ is $H^k(0)$ for some $k$.
\end{proof}

It follows that the circle dynamics of $H$ has a dense orbit $\{H^j(0)\}_{j=0}^{\infty}$. This orbit is equal to 
$\{\Pi^{-k}(1) :\: k\geq 0\}$ by Lemma~\ref{lem:8xp,1}
and hence dense by Lemma~\ref{lem:28ua,4}. This yields:

\begin{prop}\label{prop:28ua,1}
The circle dynamics of $H$ is conjugate to the linear rotation.
\end{prop}

\begin{lem}\label{lem:28ua,3}
Suppose that $\Pi$ and $\hat{\Pi}$ are presentation functions for $H$ and $\hat{H}$, respectively, both generated by maps from the ${\cal EWF}$-class. 
Then there is a unique order-preserving topological conjugacy $\upsilon$ between them which fixes $1$. 
This conjugacy is equal to the topological conjugacy between the corresponding circle maps. 
\end{lem}
\begin{proof}
Let us first show that the conjugacy $\upsilon$ between the circle mappings also conjugates between the presentation functions. On a dense subset $H^j(0)$, $\upsilon$ is 
determined by $\upsilon(H^j(0)) = {\hat H}^j(0)$ for $j>0$. But by considerations of the proof of Lemma~\ref{lem:8xp,1}, $\Pi(H^j(0)) = H^k(0)$ and 
$\hat{\Pi}(\hat{H}^j(0))=\hat{H}^k(0)$  where $k$ is the same in both cases. Then
\[ \upsilon \circ \Pi (H^j(0)) = \upsilon(H^k(0)) = {\hat H}^k(0) = \hat{\Pi}(\hat{H}^j(0)) = \hat{\Pi}\upsilon(H^j(0)) \; .\]
Since it holds on a dense set, $\upsilon$ conjugates $\Pi$ to $\hat{\Pi}$. 

For the uniqueness, observe that any conjugacy between the presentation functions normalized by $H(1)=1$ maps the set $H^{-k}(1)$ onto $\hat{H}^{-k}(0)$ for $k\geq 0$. 
For an order preserving conjugacy there is a unique such mapping. By Lemma~\ref{lem:8xp,1} it follows that the conjugacy is uniquely determined on a dense set.
\end{proof}

\subsection{Properties of orbits under complex dynamics.}
\begin{lem}\label{lem:28up,1}
For $n$ even and non-negative, define $u_n$ to be $\phi^{-1}(x_0\tau^{-n+1})$. 
For $n$ odd and positive , define $u_n$ to be $\phi_{-1}^{-1}(x_0\tau^{-n+1})$. 
For $n$ even and non-negative, consider
\[ D_n = {\cal D}(x_0,u_n) \cup {\cal D}(u_{n+1},\tau x_0)\; .\] 
Suppose that $z$ is in the domain of $H$ and $k$ is the smallest non-negative 
iterate for which $H^k(z)\in D_n$, for some even $n$. 
Then, there exists an inverse branch $H^{-k}$ defined on the connected component of $D_n$ 
which contains $H^{k}(z)$, which sends $H^k(z)$ to $z$.
\end{lem}
\begin{proof}
Since the Poincar\'{e} neighborhood is simply connected, the only
obstacle to constructing the inverse branch may be if the omitted
value $0$ is encountered. Thus suppose that for some $k' > 0$,
$\zeta$, which is an inverse branch of $H^{k-k'}$ well defined on 
the connected component $D'$ of $D_n$ 
which contains $H^{k}(z)$, maps
$H^k(z)$ to $H^{k'}(z)$ and its image contains $0$. 
First, consider the case when $D'={\cal D}(x_0,u_n)$. 
Then $H^{k-k'+1}(0)\in (0, x_0\tau^{-n+1})$.
Observe that the first entry of the iterates of $0$ by $H$ to the
interval $(0, x_0\tau^{-n+1})$ occurs by a composition of branches of $H$, which
is equal to the map $\tau^{-n}\circ \phi\circ \tau^n$, and the pullback of
$(0, x_0\tau^{-n+1})$ by this map is the interval $(x_0\tau^{-n}, x_0\tau^{-n-1})$.
Therefore, $\zeta((x_0,u_n))\subset (x_0\tau^{-n}, x_0\tau^{-n-1})$.
Using the property 4 of the definition
of ${\cal EWF}$ class, we get that $\zeta({\cal D}(x_0,u_n))\subset 
{\cal D}(x_0\tau^{-n}, x_0\tau^{-n-1})$.
In turn, $H^{-1}({\cal D}(x_0\tau^{-n}, x_0\tau^{-n-1}))=
\phi^{-1}({\cal D}(x_0\tau^{-n}, x_0\tau^{-n-1}))\cup \phi_{-1}^{-1}({\cal D}(x_0\tau^{-n}, x_0\tau^{-n-1}))$.
Now, by the property 3 of the same definition,
$$\phi^{-1}({\cal D}(x_0\tau^{-n}, x_0\tau^{-n-1}))\subset 
\phi^{-1}({\cal D}(-x_0\tau^{-n+1}, x_0\tau^{-n+1}))
\subset {\cal D}(x_0,u_n),$$ 
and
\[ \phi_{-1}^{-1}({\cal D}(x_0\tau^{-n}, x_0\tau^{-n-1}))=\tau
\phi^{-1}(({\cal D}(x_0\tau^{-n}, x_0\tau^{-n+1})) \subset \]
\[ \subset \tau {\cal D}(x_0,u_{n+2})={\cal D}(u_{n+1},\tau x_0)\; .\]
It means that $H^{k-k'-1}(z)\in D_n$, contrary to the hypothesis of the lemma. 
The remaining case $D'={\cal D}(u_{n+1}, \tau x_0)$ is very similar.
The first entry of the iterates of $0$ by $H$ to the
$(x_0\tau^{-n}, 0)$ occurs by a composition of branches of $H$, which
is equal to the map $\tau^{-n+1}\circ \phi\circ \tau^{n-1}$, 
and the pullback of
$(0, x_0\tau^{-n+1})$ by this map is the interval 
$(x_0\tau^{-n-2}, x_0\tau^{-n-1})$.
Therefore, 
$\zeta({\cal D}(u_{n+1}, \tau x_0))\subset 
{\cal D}(x_0\tau^{-n-2}, x_0\tau^{-n-1})$.
In turn, 
$$\phi^{-1}({\cal D}(x_0\tau^{-n-2}, x_0\tau^{-n-1}))
\subset \phi^{-1}({\cal D}(-x_0\tau^{-n-1}, x_0\tau^{-n-1}))
\subset {\cal D}(x_0,u_{n}),$$ 
and
$$\phi_{-1}^{-1}({\cal D}(x_0\tau^{-n-2}, x_0\tau^{-n-1}))=
\tau \phi^{-1}(({\cal D}(x_0\tau^{-n-2}, x_0\tau^{-n-3}))\subset$$
$$\tau \phi^{-1}(({\cal D}(-x_0\tau^{-n-1}, x_0\tau^{-n-1}))
\subset \tau {\cal D}(x_0,u_{n+2})={\cal D}(u_{n+1},\tau x_0).$$
As in the first case, 
it is a contradiction.
\end{proof}

\begin{defi}\label{defi:26xp,1}
Define $\Omega_{-,c}$ to be the range of the principal inverse branch of $\phi$ from the set ${\cal D}(x_0,x_0\tau)\setminus \RR_-$. 
Similarly, $\Omega_{+,c}$ is the range of the principal inverse branch of $\phi_{-1}$ from ${\cal D}(x_0,\tau x_0) \setminus \RR_+$. 
Equivalently, $\Omega_{-,c}$ is the preimage of the strip $|\Im w| < \pi$ by $\log\phi$ and likewise $\Omega_{+,c}$ is the preimage of the same strip 
by $\log\phi_{-1}$. 
\end{defi}

\begin{lem}\label{lem:28up,2}
Take a point $z$ with an infinite orbit under $H$. 
Moreover, for each $k\geq 0$, $H^k(z) \in \Omega_{-,c} \cup \Omega_{+,c}$. 
Also, suppose that the distance from the set 
$\{ H^k(z) :\: k\geq 0\}$ to $\RR$ is $0$. Then, $z\in \RR$. 
\end{lem}
\begin{proof}
By Proposition~\ref{prop:28ua,1} the dynamics of $H$ on the interval $[x_0,\tau x_0]$ is transitive and so if the orbit of $z$ accumulates on it, then 
for a subsequence $H^{k_j}(z) \rightarrow x_0$. Recall Lemma~\ref{lem:28up,1} and in particular points $u_n$ defined there. 
By Lemma~\ref{lem:28ua,1} an intersection of $\Omega_-$ with $D(x_0,\epsilon)$ for a small $\epsilon$ is contained in
the set  ${\cal D}(x_0,u_n)$ where $n$ depends on $\epsilon$ and may be made tend to infinity as $\epsilon$ tends to $0$. As a consequence, there is a sequence 
$n(j)$ tending to $\infty$ such that $H^{k_j}(z) \in D_{n(j)}$. Then by Lemma~\ref{lem:28up,1} an inverse branch of $H^{k_j}$ exists which tracks back the orbit of $z$. 
Since the orbit $z$ stays in the set $\Omega_{-,c}\cup \Omega_{+,c}$, where the only preimage of $\RR$ is inside $\RR$, this inverse branch fixes the real line. 
As the consequence of the Epstein property postulated for the ${\cal EWF}$-class, the orbit is confined to set ${\cal D}\left( H^{-k}((x_0) ,u_n(j)))\right)$ where the 
preimages are 
taken by the real dynamics. Since the real dynamics is conjugated the the golden mean circle rotation by Proposition~\ref{prop:28ua,1}, the lengths of the intervals 
$H^{-k}((x_0,u_{n(j)}))$ tend to $0$ with $j$ uniformly with respect to $k$. Thus, $z$ is in the intersection of the sequence of disks with radii tending to $0$ and centers 
on $\RR$. 
\end{proof}

\section{Quasiconformal equivalence}
\subsection{Tempering.}
Given the mappings $\phi$ and $\hat{\phi}$ from the ${\cal EWF}$ class, we would like to conjugate between their complex dynamics by a quasiconformal map. 
For technical reason, such a conjugation is difficult to obtain between the corresponding complex extensions $H$ and $\hat{H}$. We will use modified 
extensions which differ by the range, which is made smaller. 

The range of {\em tempered} $H$ will be contained in ${\cal
  D}(x_0,\tau x_0)$, but for technical reasons having to do with
quasiconformal constructions, we prefer the border to intersect the
real line at angles less than $\pi/2$. To this end, choose 
$\frac{\pi}{2} > \beta_l > \beta_r > \frac{\pi}{4}$ where $\beta_r$ is
close to $\pi/2$ and will be specified shortly.  Let $A_l$ denote
the angle $\{ z :\: |\arg (z-x_0)| < \beta_l\}$ and similarly
$A_r = \{ z:\: |\arg(z-\tau x_0) - \pi| < \beta_r\}$.   

Then 
\[ V' = {\cal D}(x_0,x_0\tau)  \cap A_l \cap A_r\; , V := V'
\setminus\{0\}\; .\]
The left branch of $H$ is $\phi$ restricted to the the preimage of
$V$, denoted $U_-$. The right branch is the preimage of $V$ by
$\phi_{-1}$ and its domain will be called $U_+$. Since $V' \subset
D(0, |\tau x_0|)$, then from the definition
of the ${\cal EWF}$-class, $U_- \subset {\cal D}(x_0,x_0/\tau)$.  

\begin{figure}
\epsfig{figure=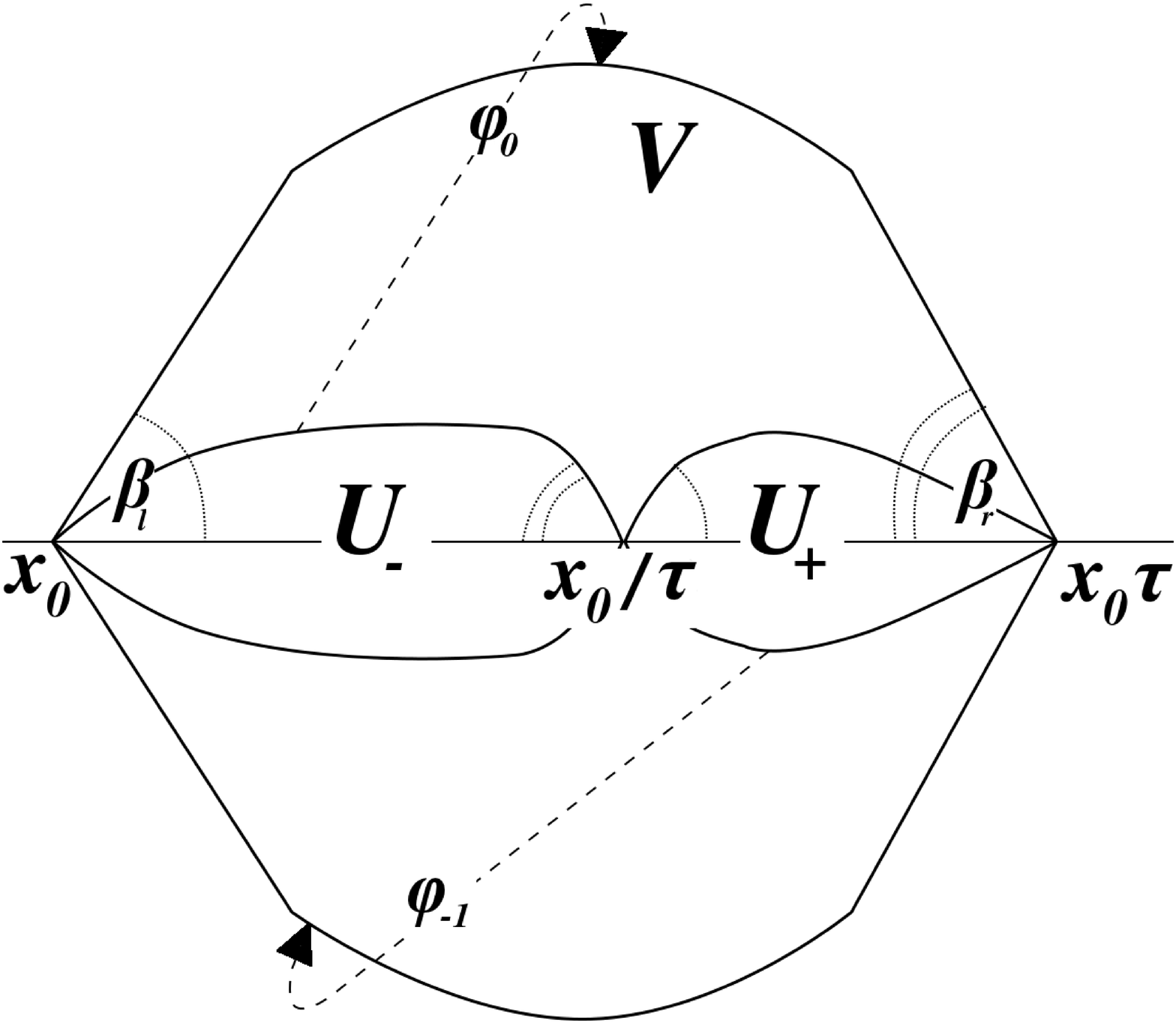, height=12cm, width=15cm}
\caption{The tempered mapping.}
\end{figure}

$H$ extends analytically through the boundary of $U_-$ at any point
with the exception of $x_0$. At $x_0$, we refer to Lemma~\ref{lem:28ua,1}.
It implies that for $\beta_r$ close enough to $\frac{\pi}{2}$, 
$U_- \subset U \subset V'$. In particular, since $\Omega_-\subset U$, the range $V'$ of the tempered map still contains 
$\Omega_-$. 

To analyze $U_+$, we need to view $\phi_{-1}$ as a rescaled version of
$\phi$. Because of that, $U_+\subset \tau U \subset {\cal D}(x_0,
x_o\tau)$. However, additionally we have Lemma~\ref{lem:28ua,1} whose
statement can be applied to $\tau U$ and means that locally $U_+$ fits
the angle $A_r$. On the other side, $U_+ \cap \RR = (x_0/\tau,x_0\tau)$ so that 
it avoids a neighborhood of $x_0$. It follows that if $\beta_r$ is
close enough to $\pi/2$, then $U_+ \subset V'$. It also follows that $\Omega_+\subset V'$. 

Finally, $\overline{U}_- \cap \overline{U}_+ = \{ x_0\tau^{-1}\}$ 
as the consequence of Lemma~\ref{lem:28ua,2}. Additionally, as a corollary from the construction,
\begin{equation}\label{equ:6xp,1}
V' \supset \Omega_-\cup\Omega_+\; .
\end{equation}

\begin{theo}\label{theo:28up,2}
Suppose that $\phi,\hat{\phi} \in {\cal EWF}$ and let $H, \hat{H}$ be their tempered complex dynamics, defined on affinely similar ranges 
$V', \hat{V}'$. Then, there is a quasiconformal automorphism $\Psi$ of the plane, symmetric w.r.t. $\RR$ and fixing $0$, $1$ which conjugates
tempered $H$ and $\hat{H}$ on their respective domains $U_-\cup U_+$, $\hat{U}_-\cup \hat{U}_+$. 
\end{theo} 

\subsection{Conjugacy between presentation functions.}
If two maps from ${\cal EWF}$ class are given, then by Lemma~\ref{lem:28ua,3}, the conjugacy between their circle dynamics 
is equal to the conjugacy between their real presentation functions. This is the overall strategy of the proof of Theorem~\ref{theo:28up,2}.
We will first construct the conjugacy between the presentation functions. 

\paragraph{Complex presentation function $\Pi$.}
Recall the definition of the presentation function on the segment
$[x_0,\tau x_0]$. 
If  $x\in [x_0,\tau^{-1} x_0]$, set $\Pi(x) = \tau x$. If $x\in
[\tau^{-1}x_0,\tau x_0]$, set $\Pi(x)=\phi_{-1}(x)$. 

$\Pi$ can also be extended to a complex map with range $V'$. The
linear branch maps $W_-:=\tau^{-1}V'$ onto $V'$. Because $\beta_r <
\beta_l$, we get $W_- \subset V'$. The non-linear branch will map only
onto $W_-\setminus \{0\}$. Let $W_+$ be $\phi^{-1}_{-1}(W_-\setminus\{0\})$. 
Taking into account the rescaling, 
\[ W_+ = \tau \phi^{-1} (\tau^{-2} V)\; .\] 
Since $\tau^{-2} V \in D(0,x_0/\tau)$, by the defining properties of
class ${\cal EWF}$, 
\[ \phi^{-1}(\tau^{-2}V) \subset {\cal D}(x_0,x_0/\tau^{2})\; .\]
Finally $W_+ \subset {\cal D}(x_0/\tau, x_0\tau)$. Invoking
Lemma~\ref{lem:28ua,1}, we see that for $\beta_r$ close enough to
$\pi/2$, $W_+\subset V'$. Moreover, $\overline{W}_+ \cap
\overline{W}_- = \{\frac{x_0}{\tau}\}$. 

\paragraph{Notational convention.}
We suppose that $\phi$ and $\hat{\phi}$ are given as in the statement of Theorem~\ref{theo:28up,2}. 
Moreover, all objects constructed from $\hat{\phi}$ (complex dynamics, presentation function, etc) will also be marked by $\hat{}$.  

\paragraph{Technical tools.}
By an {\em arc} we will mean a homeomorphic image of either an open interval or a circle. The arc is term {\em quasiconformal} if the homeomorphism can be extended to a 
quasiconformal homeomorphism of the plane. 

\begin{defi}\label{defi:29up,1} 
Let $w$ be an arc and $y \in w$. We say that $y$ is {\em sectorially accessible} 
iff there is $\epsilon>0$ and a pair of vertical angles of positive measure with the vertex at $y$ and interior denoted with $A$ so that 
$w \cap A \cap D(y,\epsilon) = \emptyset$. 
\end{defi}

\begin{fact}\label{fa:29up,1}
Let $w$ be an arc and $Y=\{y_1,\cdots,y_n\}$ a finite collection of sectorially accessible points of $w$. Then, if each connected component
of $w\setminus Y$ is a quasiconformal arc, then $w$ is quasiconformal as well.   
\end{fact}

\begin{defi}\label{defi:29up,2}
Now, suppose that $h$ is an orientation-preserving real homeomorphism of an open interval in $\RR$ (perhaps unbounded) onto its image. 
If $y$ is a point in the domain of $h$, then $h$ is said to be {\em quasi-symmetric at $y$} provided that there are $\epsilon>0$ and $K$ so that whenever
$0<|\eta|<\epsilon$, then 
\[ \frac{|h(y+\eta)-h(y)|}{|h(y-\eta)-h(y)|} \leq K \; .\]
\end{defi}

\begin{fact}\label{fa:29up,2}
If $h$ is a real orientation-preserving homeomorphism of an open
possibly unbounded interval $I$ in $\RR$ onto its image,
$Y=\{y_1,\cdots,y_n\}$ is a finite collection of points from $I$, 
$h$ is quasi-symmetric at each point from $Y$ and $h$
is quasi-symmetric on each connected component of $I\setminus Y$, 
then $h$ is quasi-symmetric on its entire domain. 
\end{fact}

\paragraph{Initial pre-conjugacy.}
\begin{prop}\label{prop:29ua,1}
There exists a quasiconformal mapping $\Psi_1$ of the plane, symmetric with respect to $\RR$, affine outside of $V'$, which satisfies the conjugacy condition 
\[ \Psi_1 \circ \Pi(z) = \hat {\Pi} \circ \Psi_1(z) \]
for $z\in \partial W_- \cup \partial W_+$. 
\end{prop}

Let $A$ denote the real and orientation-preserving affine map which transforms $V'$ onto
$\hat{V}'$. Then $A'$ is a quasiconformal map equal to $A$ outside of
$V'$ and fixing $0$. Then, on $W_-$ consider the mapping $A_- =
\hat{\tau}^{-1} \circ A' \circ \tau$. $A_1$ is quasiconformal on $W_-$
and fixes $0$. Then, $A_+$ is defined on $W_+$ as the lifting of $A_-$
to the universal covers $\phi_{-1},\hat{\phi}_{-1}$. $A_+$ is also
quasiconformal. 

Consider also the Jordan arc $w$ which consists of the boundary arcs
of $V'$, $W_-$ and $W_+$ intersected with
$\overline{\HH}_+$. $\hat{w}$ is analogous.
Additionally, a homeomorphism $u$ of $\RR$ onto
itself has been defined, which consists of $A$ outside of $(x_0,\tau
x_0)$, is equal to $A_-$ on $(x_0,\tau^{-1}x_0)$ and to $A_+$ on
$(\tau^{-1}x_0,\tau x_0)$. 

\begin{lem}\label{lem:29up,1}
$w,\hat{w}$ are quasiconformal.
\end{lem}
\begin{proof}
We will rely on Fact~\ref{fa:29up,1} and only give a proof for $w$.  Points $x_0,\tau^{-1}x_0,\tau
x_0$ are sectorially accessible by the choice of angles in the
construction of tempered dynamics and Lemma~\ref{lem:28ua,1}. The
boundary arcs of $V'$ and $W_-$ are clearly
quasiconformal. $\log\phi_{-1}$ is the Fatou
coordinate of rescaled $G$. Then, the boundary of $W_+$ is the
preimage by the Fatou coordinate of the $2\pi i$-periodic curve which
is the image of the boundary of $W_-$ under the $\log$. This curve is
quasiconformal and so is its preimage based on Fact~\ref{fa:29up,3}.    
\end{proof}

\begin{lem}\label{lem:29up,3}
The mapping $u$ is quasi-symmetric on $\RR$. 
\end{lem}
\begin{proof}
This will be based on Fact~\ref{fa:29up,2}. After removing points
$x_0,\tau^{-1}x_0,\tau x_0$, $u$ is quasi-symmetric on each connected
component. This is obvious except on $(\tau^{-1}x_0,\tau x_0)$, where
one can invoke Fact~\ref{fa:29up,3}. The quasi-symmetry at $x_0$ and
$\tau^{-1}x_0$ is also clear, so the mapping is analytic on one-sided
neighborhoods of those points. It remains to consider the
quasi-symmetry at $\tau x_0$. 

Let $G_{-1}$ denote $\tau G \tau^{-1}$. Denote $y_n =
G_{-1}^n(\tau^{-1}x_0)$, likewise in the ``hatted'' space. The
quasi-symmetry at $\tau x_0$ follows if we can show that 
\begin{enumerate}
\item
\[ \exists K>1  \forall n\geq 0\; K^{-1} \leq \frac{\tau x_0 -
  y_n}{\hat{\tau}\hat{x}_0-\hat{y}_n} \leq K\; ,\]
\item
\[ \forall n\geq 0\; \frac{\tau x_0 - y_n}{\tau x_0 - y_{n+1}} \leq
K \] and the same holds in the ``hatted'' space.
\end{enumerate}

The second statement is obvious, since the ratios tend to $1$ by the
dynamics of a neutral point of $G_{-1}$. The first statement follows
from the form of the Fatou coordinate, namely that it is
$\frac{a}{(y-\tau x_0)^{2}}$ followed by a map whose distance from the
identity is bounded. Also, $\log A_-$ moves points by bounded
distances. 
\end{proof}
 
\paragraph{Construction of $\Psi_1$.}
The union of $A$, $A_-$ and $A_+$ and $u$ already defines $\Psi_1$ in
the closure of $\HH_+$ except for the quasiconformal disk bounded by $w$.  
We can also extend the definition to the lower half-plane using the
Beurling-Ahlfors theorem, see~\cite{bahl}, based on Lemma~\ref{lem:29up,3}. Then,
extension for the region bounded by $w$ is achieved by a
quasiconformal reflection based on Lemma~\ref{lem:29up,2}. Finally,
$\Psi_1$ is defined in the lower-half plane by reflection from
$\HH_+$. The properties claimed for it in
Proposition~\ref{prop:29ua,1} are clear from the construction, so the
proof of this Proposition is finished.   

\paragraph{Conjugacy between presentation functions.}
\begin{prop}\label{prop:29up,1}
There exists a quasiconformal mapping $\Psi_2$ of the plane, symmetric with respect to $\RR$, affine outside of $V'$, which satisfies the conjugacy condition 
\[ \Psi_2 \circ \Pi(z) = \hat {\Pi} \circ \Psi_2(z) \]
for $z\in \partial W_- \cup \partial W_+$ and for $z\in (x_0,\tau
x_0)$.  
\end{prop}

To construct $\Psi_2$, we first modify $\Psi_1$ by imposing additional
conditions $\Psi_1(0)=0, \Psi_1(1)=1,
\Psi_1(\tau^{-1})=\hat{\tau}^{-1}$ without losing the properties
claimed in Proposition~\ref{prop:29ua,1}. This is easy to do, since we
can move finitely many points in side $W_-, W_+$ by quasiconformal
mappings which remain identities on the boundary.  Dynamically these
conditions mean that  $\Psi_1$ conjugates on the post-critical set of
the presentation function. We can then construct a sequence of pull-backs
$\Psi_1^n$, with $\Psi_1^0 = \Psi_1$ and $\hat{\Pi} \circ \Psi_1^{n+1}
= \Psi_1^n \circ \Pi$, defined uniquely by the condition that they map
the real domain onto itself preserving the orientation. 

By construction, $\Psi_1^n$ maps all preimages of $1$ by $\Pi$ of order not
exceeding $n$ onto the corresponding preimages of $1$ by
$\hat{\Pi}$. We can define $\Psi_2$ as any limit of the sequence
$\Psi_1^{n_k}$ using normality of quasiconformal mappings with the
same maximal dilatation. Then, $\Psi_2$ is the same as $\Psi_1$ on the
outside of the domain of $\Pi$ and additionally maps all preimages of
$1$ by $\Pi$ onto the corresponding preimages of $1$ by $\hat{\Pi}$. 
By Lemma~\ref{lem:28ua,4} this implies conjugacy on the entire real
domain $(x_0,\tau x_0)$. Proposition~\ref{prop:29up,1} has been
proved. 

\subsection{Conjugacy between tempered complex dynamics of $H,\hat{H}$.}
\paragraph{Construction of the pre-conjugacy.}
\begin{prop}\label{prop:29up,2}
There exists a quasiconformal mapping $\Psi_3$ of the plane, symmetric with respect to $\RR$, affine outside of $V'$, which satisfies the conjugacy condition 
\[ \Psi_3 \circ H(z) = \hat {H} \circ \Psi_3(z) \]
for $z\in \partial U_- \cup \partial U_+$ and for $z\in (x_0,\tau x_0)$.  
\end{prop}

Start by defining maps $B_{\pm}$. Let $B_-$ be defined on $U_-$ as the lifting of $\Psi_2$
by the universal covers $\phi, \hat{\phi}$. Since $\Psi_2$ was the
conjugacy between real presentation functions of $(x_0,\tau x_0)$ and
by Lemma~\ref{lem:28ua,3} also the conjugacy between the real dynamics
of $H,\hat{H}$, $B_-$ remains the same as $\Psi_2$ on
$(x_0,\tau^{-1}x_0)$. Additionally, it satisfies the conjugacy
condition on the boundary. Map $B_+$ is defined similarly on $U_+$ as
the lifting of $\Psi_2$ to the universal covers
$\phi_{-1},\hat{\phi}_{-1}$. By the same arguments, it is equal
to $\Psi_2$ on the real trace and satisfies the conjugacy condition on
the boundary. 

We can extend the union of $B_-$ and $B_+$ restricted to $\HH_+$ to
the lower half-plane by $\Psi_2$ and to the complement of $V'$ in the
upper half plane also by $\Psi_2$. To complete the proof of
Proposition~\ref{prop:29up,2}, it remains to extend the mapping to the
set $V'\cap\HH_+ \setminus (U_-\cup U_+)$. 

This follows by the same method as in the proof of
Proposition~\ref{prop:29ua,1}. We first consider the Jordan arc $v$
which consists of the boundary arcs of $V',U_-,U_+$ in the upper half
plane and the analogous arc $\hat{v}$. 

\begin{lem}\label{lem:29up,4}
Arcs $v,\hat{v}$ are quasiconformal.
\end{lem}
\begin{proof}
The proof is very similar to the proof of Lemma~\ref{lem:29up,1} and
omitted.
\end{proof}

Now $\Psi_3$ is constructed on the region bounded by $v$ by the
quasiconformal reflection and finally extended to the lower
half-plane by reflection. 

\paragraph{Proof of Theorem~\ref{theo:28up,2}.}
Thus, we construct a sequence of
quasi-conformal homeomorphisms $\Psi^n$ of the plane, by setting
$\Psi^0 = \Psi_3$ and defining $\Psi^n$ for $n>0$ as $\Psi^{n-1}$
outside of $U_+\cup U_-$ and to be the lifting of $\Psi^{n-1}$ to the
universal covers $H_{|U_+}, \hat{H}_{|\hat{U}_+}$ and $H_{|U_-},
\hat{H}_{|\hat{U}_-}$. Both liftings are uniquely defined by the
requirement that $\Psi^n$ should fix the real line with its
orientation. 

The sequence $\Psi^n(z)$ actually stabilizes for every $z\notin
\tilde{K}_H$. By Lemma~\ref{lem:28up,7}, $\tilde{K}_H = K_H$. 
So $\Psi^n$ converge on the complement of $K_H$ and by taking a
subsequence can be made to converge globally to some map
$\Psi^{\infty}$. Outside of $K_H$, $\Psi^{\infty}$ satisfies the
functional equation $\Psi^{\infty} H = \hat{H} \Psi^{\infty}$ and then
it also satisfies it on $K_H$ by continuity, since $K_H$ has an empty
interior by Lemma~\ref{lem:29up,6}. So we can set $\Psi := \Psi^{\infty}$ and this
concludes the proof of Theorem~\ref{theo:28up,2}.

\section{Julia sets and Towers} 
\subsection{Julia sets.}
We consider the filled-in Julia set for the untempered dynamics $H$. It can be constructed in several stages. 
First $K'_H = \{ z:\: \Omega_-\cup\Omega_+ :\: \forall j>0\; H^j(z) \in \Omega_-\cup \Omega_+ \}$. 
Then, $K''_H = \{ x_0, \tau x_0 \}$. 
Then, $K'''_H = \{ x\in \Omega_-\cup\Omega_+ :\: \exists j\geq 0\; 
H^j(x) = x_0/\tau \}$. 
By definition, $K_H = K'_H \cup K''_H \cup K'''_H$. 

At the end of this section we will prove the following theorem:
\begin{theo}\label{theo:28up,1}
\[ K_H = \overline{\{ H^{-j}(x_0/\tau) :\: j\geq 0\}} \; .\]
\end{theo}

As a corollary, we get
\begin{lem}\label{lem:29up,6}
Set $K_H$ has an empty interior.
\end{lem}
\begin{proof}
If the interior of $K_H$ is not empty, then by
Theorem~\ref{theo:28up,1}, the interior of $K_H$ contains a
neighborhood of $x_0/\tau$. But any neighborhood of that point sticks
out of the domain of definition for $H$. 
\end{proof}

\paragraph{The Julia set for tempered dynamics.}
By analogy to the set $K_H$, we consider the filled-in Julia set
$\tilde{K}_H$ for the tempered dynamics. It can be constructed in several stages. 
First $\tilde{K}'_H = \{ z:\: U_-\cup U_+ :\: \forall j>0\; H^j(z) \in U_-\cup U_+ \}$. 
Then, $\tilde{K}''_H = \{ x_0, \tau x_0 \}$. 
Then, $\tilde{K}'''_H = 
\{ x\in U_-\cup U_+ :\: \exists j\geq 0\; H^j(x) = x_0/\tau \}$. 
By definition, $\tilde{K}_H = \tilde{K}'_H \cup \tilde{K}''_H \cup \tilde{K}'''_H$. 

\begin{lem}\label{lem:28up,7}
For any dynamics generated by the ${\cal EWF}$ class, $\tilde{K}_H =K_H$.
\end{lem}
\begin{proof}
Since the tempered dynamics is a restriction of the untempered version to a smaller domain, ${\tilde K}_H \subset K_H$. 
In order to prove the opposite inclusion it will be enough to show that $K_H \subset U_- \cup U_+$. This is because $K_H$ is invariant. 
Recall that $H$ maps any point of $\Omega_-\cup\Omega_+$ not in 
$U_-\cup U_+$ into ${\cal D}(x_0,\tau x_0) \setminus V'$. By inclusion~(\ref{equ:6xp,1}) this is outside of the domain of untempered dynamics and hence such points cannot 
belong to $K_H$. 
\end{proof}

\subsection{Dynamics in towers.}
Let $H$ be the complex dynamics, tempered or not, generated by some $\phi$
in the ${\cal EWF}$-class.

\begin{defi}\label{defi:11fa,1}
Define, for $n=0,1,2,...,$
$H_{-n}(z)=\tau^n H(z/\tau^n)$.
Then $\tau^n K_H$ is the Julia set of the map
$H_{-n}:\Omega_{-n} \to D^{*}_{-n}$, where $\Omega_{-n}=\tau^n(\Omega_{-}\cup
\Omega_{+}), D^{*}_{-n}=\tau^n {\cal D}(x_0,\tau x_0) \setminus \{0\}$. 

The collection of maps $H_n:U_n\to V_n$, $n=0,-1,...$
forms the {\it tower} of $H$, tempered or untempered, respectively. 
Map $H_n$ will be referred to as the
$n$-th {\em level} of the tower.
\end{defi}

\begin{lem}\label{lem:30ua,1}
For any $0\geq m >n$, each branch of untempered $H_m$ on its domain is a composition for
branches of $H_n$. This includes an assertion that for each component
of $\Omega_{m}$ there is a
particular composition of branches on $H_n$ which is well defined on
this entire component. 
\end{lem}
\begin{proof}
By induction, it is enough to prove this statement when $n=m-1$. By
rescaling, we can reduce the situation to $m=0, n=-1$. 

Then, one branch of $H_0$ is $\phi_{-1}$ defined on $\Omega_+$. It is
the same as a branch of $H_{-1}$ defined on $\tau\Omega_-$. All we need
to check is $\Omega_+ \subset \tau\Omega_{-}$. By definition, however, 
$\Omega_+$ is the preimage by $\phi_{-1} = \tau \phi \tau^{-1}$ of
${\cal D}^{*}(x_0,\tau x_0)$ which is $\tau\phi^{-1}\left({\cal  D}^{*}(x_0,\tau^{-1}x_0) \right)$, which is clearly inside
    $\tau\Omega_-$. 

The second branch of $H_0$ is a composition $\phi_{-2} \circ
\phi_{-1}$, both of which are branches on $H_{-1}$. We need, however,
to check the inclusions between domains. $\phi_{-1}$ as a branch of
$H_{-1}$ is defined on $\tau\Omega_-$. So the next inclusion to check
is $\Omega_- \subset \tau\Omega_-$. 

This follows from a sequence of inclusions:
\[ \Omega_- = \phi^{-1}({\cal D}^*(x_0,\tau x_0)) \supset \phi^{-1}({\cal
  D}(x_0\tau^{-1},\tau x_0)) = \phi^{-1}(\tau{\cal
  D}(x_0,x_0\tau^{-2})) \supset \]
\[ \supset \phi^{-1}\left( \tau\phi^{-1}({\cal
  D}^*(x_0\tau^{-1}, x_0\tau^{-2})) \right) =
(\phi^{-1}\circ\tau\circ\phi^{-1}\circ\tau^{-2}){\cal D}^*(x_0,\tau x_0)
\; .\] 

But 
\[ \phi^{-1}\circ\tau\circ\phi^{-1}\circ\tau^{-2} = \tau^{-1}\phi_{-1}^{-1}\circ\phi_{-2}^{-1} = \tau^{-1} \phi^{-1} \]
by Lemma~\ref{lem:9xa,1}. This proves $\Omega_- \supset
\tau^{-1}\Omega_-$. 

Next, we need to check that $\phi_{-1}(\Omega_-) = \tau \phi
(\tau^{-1}\Omega_-) \subset \tau^2\Omega_-$, where the last set is the
domain of $\phi_{-2}$. This inclusion is equivalent to 
\[ \tau^{-1} \phi \tau^{-1}(\Omega_-) = G(\Omega_-) \subset \Omega_-\;
.\]

The last inclusion follows from 
\[ \tau^{-2}\phi = \phi\circ G \] 
which implies that $G$ maps the domain of $\phi$ into itself. 
\end{proof}

\begin{lem}\label{lem:30up,1}
The assertion of Lemma~\ref{lem:30ua,1} remains true for the dynamics
on tempered domains if the parameter $\beta_r$ in the definition of
tempered domain is close enough to $\pi/2$.  That is, $H_n$ considered on a component $\tau^n
U_{\pm}$ of its domain can be realized as a composition of branches of
$H_m$, also defined on their tempered domains $\tau^m U_{\pm}$. 
\end{lem}
\begin{proof}
We only need to check that inclusions between the domains stated in the
proof of Lemma~\ref{lem:30ua,1} remain true for tempered domains. 

The first inclusion is $U_+ \subset \tau U_-$. $U_+$ was defined as 
the preimage by $\phi_{-1}$ of $\tau^{-1}V$. Therefore, 
\[ U_+ = \tau \phi^{-1}(\tau^{-2} V) \subset \tau\phi^{-1}(V) = \tau U_- \]
as needed. 

Next, we need $U_- \subset \tau U_-$. 
The boundaries of the corresponding untempered domains intersect only at $x_0$. In a
neighborhood of this point $\tau U_-$ is the preimage by a real conformal map of an angle 
${\cal A}_r$. Since $\beta_r$ is chosen greater than $\frac{\pi}{4}$ and in view of Lemma~\ref{lem:28ua,1}, the desired inclusion holds 
on a sufficiently small neighborhood of $x_0$. Outside of this neighborhood, the distance between the boundaries of 
the untempered domains is positive. By picking $\beta_r$ sufficiently close to $\pi/2$, we can move the border of $\tau U_-$ arbitrarily little from 
the border of $\tau \Omega_-$ and so preserve the inclusion. 

Thirdly, we need $U_+=\phi^{-1}_{-1}(U_-) \subset \tau^2 U_-$. But we have already seen that 
$U_+ \subset \tau U_-$ and $U_- \subset \tau U_-$. After scaling by $\tau$, 
the desired inclusion follows. 
\end{proof}

\begin{prop}\label{prop:11fa,1}
For every  $z\in \CC$ which is never mapped to $\RR$ by the untempered
tower dynamics, there exist
sequences $z_n \in \CC$ and $m_n \in \NN\cup \{0\}$, $n=0,1,\cdots$ chosen so 
that $z_0 = z$, $z_{n-1}$ belongs to the domain of $H_{-m_{n-1}}$ and $z_n$ 
is an image of
$H_{-m_{n-1}}(z_{n-1})$ by the tower dynamics. Furthermore, for some $\eta>0$  and every $n>0$
\begin{itemize}
\item
$\dist (z_n, \RR) > \eta |\tau|^{m_n}$ with $\dist$ meaning the Euclidean 
distance, {\em or}
\item
\[ \tau^{-m_n} z_n \in (\Omega_-\cup \Omega_+) \setminus
(\Omega_{+,c} \cup \Omega_{-,c})\; \]
(cf. Definition~\ref{defi:26xp,1}.) 
\end{itemize}
For every $z$, one can find a sequence $(n_p)$
and choose one case of the alternative, so that the claim of this
Proposition holds for all $n:=n_p$ with this particular case.    
\end{prop}

\paragraph{Auxiliary dynamics of $\Gamma$.}
For a point $\tau^n x_0$ the highest level tower map defined on a
neighborhood of this point is $H_{-n-2}$. We have
\begin{equation}\label{equ:30va,1}
 H_{-n-2}(\tau^n x_0) = \tau^{n+2} H(\tau^{-2}2 x_0) \tau^{n+2}
\tau^{-1} x_0 = \tau^{n+1} x_0\; .
\end{equation}

This leads one to consider the map 
\[ \Gamma(z) = \tau^{-1} H_{-2}(z) \; .\]
Then, formula~(\ref{equ:30va,1}) means that $\tau^n x_0$ is a fixed
point for $\Gamma_{-n}=\tau^n\Gamma \tau^{-n}$, in particular, 
$x_0$ is a fixed point of $\Gamma$.  

To find out the local dynamics at this point, observe that 
\[ \Gamma \circ \Gamma = \tau^{-1} H_{-2} \tau^{-1} H_{-2} \tau
\tau^{-1} = \tau^{-1} H_{-2} H_{-1} \tau^{-1} = \tau^{-1} H \tau^{-1}
= G\; .\]

However, $\Gamma$ is orientation-reversing on the real line, so
$\Gamma'(x_0)=-1$. As a consequence of the defining properties of the
${\cal EWF}$-class, $x_0$ is the global attractor for $\Gamma^{-1}$ on
$\CC\setminus \RR$.  

\begin{lem}\label{lem:30va,1}
Let $\Psi_k$ denote the composition of $k$ maps in the form 
$H_{-k-1} \circ \cdots \circ H_{-2}$ defined on a neighborhood of
$x_0$. We put $\Psi_0 = \mbox{id}$. Choose $z\in\CC$ so that 
$\arg (z-x_0) \mod \pi \in [\frac{\pi}{5},\frac{4\pi}{5}]$. 
There exists $\epsilon_0>0$ such that for every $\epsilon_0 \geq
\epsilon > 0$ and $z$ as above there is $k\geq 0$ so that 
$\arg(\Psi_k(z) - \tau^k x_0)\in [\frac{\pi}{5}, \frac{4\pi}{5}]$,
$|\Psi_k(z)-\tau^k x_0| < 2\epsilon |\tau|^k$
and either 
$|\Psi_k(z)-\tau^k x_0|\geq |\tau|^k \epsilon$, or $\Psi_{k+1}(z)$ is
defined and 
\[ \arg(\Psi_{k+1}(z)-\tau^{k+1}x_0)\mod \pi \in
(\frac{\pi}{10},\frac{\pi}{5}) \cup (\frac{4\pi}{5},\frac{9\pi}{10})
\; .\]
\end{lem}
\begin{proof}
One easily shows by induction that, in a generalization of
formula~(\ref{equ:30va,1}), $\tau^{-k}\Psi_k = \Gamma^k$. Then
$\epsilon_0$ should be chosen so that $\Gamma$ is defined on 
$D(x_0,\epsilon_0)$.

Since
$\Gamma^{-1}$ attracts to $x_0$ on every angle disjoint from the real
line, then the orbit $\Gamma^k(z)$ for $z$ in this angle has to either
leave the angle, or leave $D(x_0,\epsilon)$. Then choose $k$ to the
first moment when that happens. If $|\Gamma^k(z)-x_0| \geq \epsilon$,
this immediately translates by rescaling to the first case of this
Lemma. The upper bound $|\Gamma^k(z) - x_0| < 2\epsilon$ can be
obtained by specifying $\epsilon_0$ small enough, since
$\Gamma'(x_0)=-1$. 

If $\Gamma^k(z) - x_0$ is still inside the angle, then we can
reason by the conformality of $\Gamma$ at $x_0$ and, by decreasing
$\epsilon_0$ if needed, get that $\Gamma^{k+1}(z)-x_0$ is still in a
slightly larger angle, leading to the second case in the Lemma.
\end{proof}

\paragraph{Proof of Proposition~\ref{prop:11fa,1}}
Suppose $z_{n-1}$ has been chosen in the domain of $H_{-m_{n-1}}$. Let 
$z=H_{-m_{n-1}}(z_{n-1})$.  The
first possibility is that $z_{n-1}$ is in the Julia set of
$H_{m_{n-1}}$. This case is resolved by applying
Lemma~\ref{lem:28up,2}. If $H^k_{-m_{n-1}}(z)$ gets away from the
real line, or outside of $\Omega_{-c,} \cup \Omega_{+,c}$ for any
$k>0$, then we can
set $z_n = H^k_{-m_{n-1}}$ and $m_n:=m_{n-1}+1$. One of those must
occur, since otherwise by Lemma~\ref{lem:28up,2}, $z\in\RR$ contrary
to the hypothesis of Proposition~\ref{prop:11fa,1}.

If $z_{n-1}$ is not in the Julia set, then we map
$z=H_{-m_{n-1}}(z_{n-1})$ by the
dynamics of $H_{-m_{n-1}}$ until the first moment $q$ when $w :=
H^q_{m_{n-1}}(z)$ is no longer in the domain of $H_{-m_{n-1}}$, or $w$
gets away from the real line by $\eta |\tau|^{m_{n-1}}$. If the second
possibility occurs, then we set $z_n=w$ and proceed like in the
previous case. Otherwise, $w$ is outside the domain of
$H_{-m_{n-1}}$ and still close to the real line, thus in a neighborhood 
of one of the points where the domain of $H_{-m_{n-1}}$ touches
$\RR$, that is $x_0$, $\tau x_0$, and $x_0/\tau$ rescaled by
$\tau^{m_{n-1}}$. The size of those neighborhoods can be controlled by
setting $\eta$ small enough.  

In other words, $\tau^{-m} w \in D(x_0,\epsilon)$ where
$m=m_{n-1},m_{n-1}+1$ or $m_{n-1}-1$ and $\epsilon$ is small depending
on $\eta$.   

We then apply maps $\Psi_k$ from
Lemma~\ref{lem:30va,1} to $\tau^{-m}w$ obtaining the sequence 
$w_k=\tau^m\Psi_k(\tau^{-m}x)$.  

The first possibility in Lemma~\ref{lem:30va,1} means that $w_k$ is
inside a fixed angle disjoint from $\RR$ in a distance greater than
$\epsilon|\tau|^{m+k}$ from $x_0\tau^{m+k}$. So, its distance from $\RR$
is bounded away from $0$ by $\eta |\tau|^{m+k}$, perhaps after
decreasing $\eta$ again. At the same time, its distance from $x_0
\tau^{m+k}$ is also bounded above by $2\epsilon|\tau|^{m+k}$, so by choosing
$\epsilon$ suitably small, $w_k$ belong to the domain of
$H_{-m-k-2}$. So, in this case we set $z_n = w_k$ and $m_n = m+k+2$. 
  
The second possibility is that $\arg(w_k - x_0\tau^{m+k})\mod\pi$ is
either between $\pi/10$ and $\pi/5$, or between $4\pi/5$ and
$9\pi/10$. The boundary of $\Omega_{-,c}$ is tangent to $\RR$ at
$x_0$ and likewise the boundary of $\Omega_{+,c}$ is tangent to $\RR$
at $\tau x_0$. Since on the other hand, the boundaries of
$\Omega_+,\Omega_-$ intersect $\RR$ at angle $\pi/4$ by choosing 
$\epsilon$ suitably small we can guarantee
that 
\[ \tau^{-m-n}w_k \in \left(\Omega_-\cup \Omega_+) \right) \setminus
\left( \Omega_{-,c} \cup \Omega_{+,c} \right) \]
leading to the second case in Proposition~\ref{prop:11fa,1} with $z_n=w_k$
and $m_n = m+k$. 

Proposition~\ref{prop:11fa,1} follows. 

\subsection{Hyperbolic Metric}
\begin{defi}\label{defi:8xa,1}
Let $\rho_n$ be the Poincar\'{e} metric on ${\cal D}(\tau^nx_0,\tau^{n+1}x_0)\setminus \RR$. 
Since this set is disconnected, it means the Poincar\'{e} metric on each connected component, with infinite distance between the components. 
Similarly, $\rho_{\infty}$ is the Poincar\'{e} metric of $\CC\setminus\RR$.
\end{defi}

Note that $\rho_\infty$ is invariant under the rescaling $z\mapsto \tau z$. 

If $\rho$ is a metric and $F$ a function, we will write $D_{\rho}
F(z)$ for the expansion ratio with respect to the metric $\rho$, thus 
\[ D_{\rho} F(z) = \frac{|F^* d\rho(z)|}{|d\rho(z)|} \; .\]

By Schwarz's lemma, we have $D_{\rho_{\infty}} H(z) > 1$ for every $z \in \Omega_+
\cup \Omega_-\setminus H^{-1}(\RR)$. 

We will observe expansion of the hyperbolic metric based on the
following fact:  

\begin{fact}\label{fa:20gp,2}
Let $X$ and $Y$ be
hyperbolic regions and $Y\subset X$ and $z\in Y$. Let $\rho_X$ and
$\rho_Y$ be the hyperbolic metrics of $X$ and $Y$, respectively. 
Suppose that the hyperbolic
distance in $X$ from $z$ to $X\setminus Y$ is no more than $D$. 
For every $D$ there
is $\lambda_0 > 1$ so that  
$|\iota'(z)|_H\leq \frac{1}{\lambda_0}$, where $\iota :\: Y \rightarrow X$
is the inclusion, and the derivative is taken with respect to
the hyperbolic metrics in $Y$ and $X$, respectively.
\end{fact}

The following lemma is stated in terms of $H$, but clearly it applies
to any $H_k$ as well, because the only difference is the conjugation
by a power of $\tau$, which is the isometry of the hyperbolic metrics
involved. 

\begin{lem}\label{lem:5ga,1} 
For any $z\in (\Omega_-\cup \Omega_+) \setminus H^{-1}(\RR)$, we get 
that the hyperbolic metric expansion ratio 
\[ \frac{|H^*d\rho_{\infty}|}{|d\rho_{\infty}|}(z) \geq
\frac{|d\rho'|}{|d\rho_{\infty}|} = D^{-1}_{\iota}(z)\]
where $\rho'$ is the hyperbolic metric of $\CC \setminus
H^{-1}(\RR)$ and $\iota$ is the inclusion map from $\CC \setminus
H^{-1}(\RR)$ into $\CC\setminus\RR$.  
\end{lem}
\begin{proof}
The second equality is obvious and the claim is equivalent to 
\begin{equation}\label{equ:30vp,1}
\frac{|H^*d\rho_{\infty}|}{|d\rho'|}(z) \geq 1\; . 
\end{equation}

By Lemma~\ref{lem:30ua,1}, for every $k>0$ we can represent both
branches of $H$ as compositions of $H_{-k}$ defined on appropriate
sets. For a composition of $H_{-k}$ each inverse branch $\zeta$ can be
defined from $\HH_+$
or $\HH_-$ intersected with $D(0,\tau^k R)$ is univalent onto some set 
$W_{\zeta}$ which is contained
in $\tau^k (\Omega_-\cup\Omega_+)$. Domains $W_{\zeta}$ cover the
domain of $H$ and $\zeta$ are analytic continuations of inverse branches of  
$H$. If $\rho_W$ denotes the hyperbolic metric of $W$ and $\rho_k$ the
hyperbolic metric of $D(0,\tau^k R)$, then $|d\rho_W| =
|H^*d\rho_k|$. Since $W \subset (\CC\setminus \RR) \setminus
H^{-1}(\RR)$, $|d\rho_W| \geq |d\rho'|$ and 
\[ \frac{|H^*d\rho_{k}|}{|d\rho'|}(z) \geq 1\; .\]
But when $k\rightarrow\infty$, then $|d\rho_k|\rightarrow
|d\rho_{\infty}|$ from which estimate~(\ref{equ:30vp,1}) follows.
\end{proof}

We will write $D_{\rho_{\infty}}f =
\frac{|f^*d\rho_{\infty}|}{|d\rho_{\infty}|}$. 

\begin{lem}\label{lem:8xa,1}
Suppose that $z_0$ belongs to $(\Omega_- \cup \Omega_+)\setminus (\Omega_{c,-}\cup \Omega_{c,+})$, see Definition~\ref{defi:26xp,1}.  
Then for every $M>0$ there exists $r>0$ so that the image of the hyperbolic ball centered at $z_0$ with radius $r$ with respect to 
the metric $\rho_0$ by the branch of $H$ which is defined at $z_0$ contains the set $W_M = \{ w\in {\cal D}^*(x_0,\tau x_0) :\: |\Re\log\frac{w}{H(z_0)}| < M\}$. 
\end{lem}
\begin{proof}
To fix attention, the $z_0 \in \Omega_-$. 
For this proof, we will restrict $H$ to $\Omega_-$. Furthermore, from the definition of the ${\cal EWF}$ class, $H$ can be extended analytically as a map from some set 
$U\subset {\cal D}(x_0,\tau x_0)$  onto $D(0, R)\setminus \{0\}$ with $R>\tau x_0$. 
 
Recall that $h = \log H$ can be well defined at the lifting of $\exp$ to the universal cover of ${\cal D}^*(x_0,\tau x_0)$ by $H$. We will choose the lifting so that $h$ is 
symmetric with respect to the real line. $h$ is univalent and non-contracting from $\rho_0$ to the hyperbolic metric $\rho'$ of $\HH_+$ or $\HH_-$ intersected with 
$\{ w :\: \Re w < \log R\}$. By the hypothesis of the Lemma, $|\Im h(z_0)| \geq \pi$. To fix attention, let $\Im h(z_0) = a > 0$. 
Then the diameter of the set $\log W_M \cap \{ w :\: a-\frac{\pi}{2} \leq \Im w \leq a+\frac{3\pi}{2} \}$ with respect to $\rho'$ is bounded depending only on $M$ and $r$ can be chosen equal to this bound.
\end{proof}

\paragraph{Uniform expansion.} 
Now take any point $z\in \CC$ which is never mapped to $\RR$ by the
untempered tower dynamics. 
Proposition~\ref{prop:11fa,1} then delivers a sequence $z_n$. 
Let $\chi_n$ be the corresponding tower iterate which maps $z$
to $z_n$.

\begin{lem}\label{lem:5gp,1}
For every $D$ there exists $\lambda >
1$, such that for every $n$ and every $w$ in the ball centered at $z_n$
with radius $D$ with respect to $\rho_{\infty}$,
$D_{\rho_{\infty}}H_{-m_n}(w) > \lambda$, provided that $w$ is in the
domain of $H_{-m_n}$. 
\end{lem}
\begin{proof}
By Lemma~\ref{lem:5ga,1} and
Fact~\ref{fa:20gp,2}, $DH_{\rho_{\infty}}(w) > \lambda(w) > 1$ where $\lambda(w)$
depends only on the distance in $\rho_{\infty}$ from $w$ to
$H^{-1}_{-m_n}(\RR)$. But if
either case of the alternative in Proposition~\ref{prop:11fa,1} holds,
then points $z_n$ are
all in a uniformly bounded $\rho_{\infty}$-distance from the
corresponding set $H_{m_n}^{-1}(\RR)$. This is clearly true in the
first case of the alternative. In the second case $z_n$ we will apply Lemma~\ref{lem:8xa,1} 
with the obvious rescaling by $\tau^{-m_n}$,  with $z_0 = z_{m_n}$ and any fixed positive $M$, since any ring will intersect the real line.  
We get a bounded hyperbolic distance from $z_n$ to $H_{-m_n}^{-1}(\RR)$ and also from $w$ to  $H_{-m_n}^{-1}(\RR)$ by the triangle inequality.
The claim follows from Lemma~\ref{lem:5ga,1} and Fact~\ref{fa:20gp,2}.  
\end{proof}

\begin{lem}\label{lem:5gp,2}
For every $n$, let $\zeta_n$ denote the inverse branch of $\chi_n$ which
maps $z_n$ to $z$ defined on some simply-connected set $U_n \ni z_n$. 
 Then for every $D$ and $\varepsilon$ there exists
$n_0$ such that for every $n\geq n_0$ if the diameter of $U_n$ with
respect to $\rho_{\infty}$ does not exceed $D$, then $\zeta_n(U_n)$ is
inside the hyperbolic ball of radius $\varepsilon$ centered at $z$. 
\end{lem}
\begin{proof}
Pulling back a $U_n$ will not increase its
diameter, so each time we pass $z_m$ its radius will be shrunk by a
definite factor.
\end{proof}

\paragraph{Proof of Theorem~\ref{theo:28up,1}.}
Let $z_0 \in K_H$. If the orbit of $z_0$ by $H$ ever enters the real line, then the claim of the Theorem follows 
from Proposition~\ref{prop:28ua,1}. If not, recall Lemma~\ref{lem:28up,2}. If the distance from the orbit of $z_0$ to $\RR$ is positive, then quite obviously
that $\rho_0(H^n(z_0),H^{-1}((x_0,\tau x_0)))$ is bounded independently of $n$. In the remaining case, by Lemma~\ref{lem:28up,2} we get a sequence $m_n$ such that $z_{m_n}$ is not in $\Omega_{-,c}\cup\Omega_{+,c}$. Then we use Lemma~\ref{lem:8xa,1} with $z_0 := z_{m_n}$ and any $M$ fixed and positive to get that 
$\rho_0(H^{m_n}(z_0),H^{-1}((x_0,\tau x_0)))$ is bounded independently on $n$. So, setting $m_n := n$ in the first case, we can now proceed in a uniform way. 
 
By a reasoning used in the proof of Lemma~\ref{lem:5gp,2}, we get that the inverse 
branch $H^{-{m_n}}$ which tracks back the orbit of $z_0$ shrinks the metric $\rho_0$ at a fixed rate. Thus, $\rho_0(z_0,H^{-{m_n}}((x_0,\tau x_0)))$ shrinks exponentially fast with $n$. Finally, preimages of $\tau^{-1} x_0$ are dense in $(x_0,\tau x_0)$ by Proposition~\ref{prop:28ua,1}.

\paragraph{Density of the Julia sets.}
\begin{prop}\label{prop:11fa,3}
In the tower of the untempered complex dynamics 
$H$ for every {\cal EWF}-map, the Julia set
\[ \bigcup_{n=1}^{\infty} \tau^n K_{H} \]
is dense in $\CC$. 
\end{prop}

We can now prove Proposition~\ref{prop:11fa,3}. For some fixed $D$ and
every $n$, we can find 
an element of $H^{-1}_{-m_n}(\RR)$, moreover, a preimage of
$0$ by $H_{-m_n}$,  which can be joined 
to $z_n$ by a simple arc $\gamma_n$ of hyperbolic length which does not exceed 
some fixed $D$ and which is completely contained in $\tau^{m_n}
(\Omega_-\cup  \Omega_+)$. This follows from Lemma~\ref{lem:8xa,1} after rescaling. 
We can then find $k$ which is at least equal to $m_n$
and large enough so that the tower iterate $\chi_n$ can be represented
as an iterate of $H_{-k}$. 

Then the inverse branch $\zeta_n$ is defined on a neighborhood of
$\gamma_n$. We can apply Lemma~\ref{lem:5gp,2}  to get that $\zeta_n$ maps
$\gamma_n$ into a neighborhood of $z$ whose diameter shrinks to $0$ as
$n$ grows. Letting $n$ go to
$\infty$, we get that every ball
centered $z$ contains a preimage of $0$ by some iterate of the tower dynamics.
But every preimage of $0$ in the tower belongs to some
$K_{H_k}$ and so Proposition~\ref{prop:11fa,3} follows.

\section{Rigidity of Towers}
\subsection{Conjugacy between towers.}
Given the untempered dynamics and their towers  built for two ${\cal EWF}$-maps, 
$H$ and $\hat{H}$, we initially construct a
quasiconformal conjugacy {\em between the towers}.

\begin{prop}\label{prop:17fa,1}
There is a quasi-conformal homeomorphism
$\Upsilon$ of the plane, symmetric w.r.t. the real axis,
and normalized so that $\Upsilon(0)=0, \Upsilon(1)=1, \Upsilon(\infty)=\infty$,
which  conjugates every $H_{-n}$ with $\hat{H}_{-n}$:
$\Upsilon\circ H_{-n}=\hat H_{-n}\circ \Upsilon$ whenever both sides
are defined.
Moreover, $\Upsilon(z)=\hat \tau \Upsilon(z/\tau)$ for any $z\in {\C}$.
\end{prop}

We initially construct the conjugacy between the {\em tempered}
towers, $H^t_{-n}, \hat{H}^t_{-n}$. 
The conjugacy $\Upsilon^0$ is between $H^t_0$ and $\hat{H}^t_0$ is
obtained by Theorem~\ref{theo:28up,2}. Denote $\Upsilon^n(z) = \hat\tau^n
\Upsilon^0(\tau^{-n}z)$. For every $n$, we have 
\[ \Upsilon^n H^t_{-n}(z) = \hat\tau^n \Upsilon_0 (\tau^{-n}\tau^n H^t(\tau^{-n} z)) = 
\hat\tau^n \hat{H}^t(\Upsilon_0(\tau^{-n} z)) = \hat{H}^t_{-n}(\Upsilon^n(z)) \]
and so $\Upsilon^n$ conjugates $H^t_{-n}$ to $\hat{H}^t_{-n}$. By
Lemma~\ref{lem:30up,1}, the same $\Upsilon^n$ also conjugates $H^t_{-i}$ to
$\hat{H}^t_{-i}$ for $i=0,\cdots,n$. 

Using the compactness of the family $\Upsilon^n$, we pick a limit point
$\Upsilon$ which conjugates the whole towers. 

Two things need to be checked: that $\Upsilon$ also conjugates between
{\em untempered} towers and that it is invariant under the rescaling. 

\paragraph{Uniqueness of the conjugacy on the Julia set.}
\begin{lem}\label{lem:18fa,1}
Suppose that $H$ is a tempered dynamics built for some ${\cal EWF}$-map.
Let $\Upsilon$ be a homeomorphism which self-conjugates $H$,
i.e. $\Upsilon(H(z)) = H(\Upsilon(z))$ for every $z\in U_-\cup
U_+$. In addition, $\Upsilon$ fixes $0$, is symmetric about the real line and
preserves its orientation. 
Then $\Upsilon(z) = z$ for every $z\in \tilde{K}_{H}$. 
\end{lem}
\begin{proof}
Let $I_0 = (x_0,\tau x_0)$. $\overline{I}_0$ is the domain of circle dynamics
generated by $H$. By Proposition~\ref{prop:28ua,1}, the orbit of $0$
is dense in this interval, and so $H(z)=z$ for all $z\in I_0$. Now
consider the set ${\cal I} = \cup_{k=0}^{\infty}
H^{-k}(\overline{I}_0)$. We distinguish regular points mapped into
$I_0$ by $H^k$ for some $k\geq 0$. Regular points constitute a disjoint
union of arcs, each of which is mapped onto $I$ by some iterate of
$H$. Branching points are preimages of $x_0$ and $\tau x_0$ and arcs
of regular points join there. 

Let $z_0$ be regular point and $z_1=\Upsilon(z_0)$. If $z_1$ and $z_0$
belong to the same regular arc, then $z_1 = z_0$ since $H^k$ is
one-to-one on regular arcs and $\Upsilon$ is the identity of $I_0$. 

Point $z_1$ cannot be a branching point, since the conjugacy must permute
branching points among themselves. If it belongs to a different
regular arc, then we construct an isotopy $\Upsilon_t$ between $\Upsilon$ and
$\mbox{id}$ by scaling the Beltrami coefficient of $\Upsilon$ by
$t$. Each $\Upsilon_t$ conjugates $H$ to some $H_t$ which is a complex
dynamics symmetric about $\RR$. ${\cal I}_t = \Upsilon_t({\cal I})$ is
the full preimage of the domain of the circle dynamics for $H_t$ and
in particular, $\Upsilon_t$ still permutes among the branching points
of ${\cal I}_t$.    
Then $z_t$ has to join $z_1$ to $z_0$ inside $\Upsilon_t({\cal
  I})$. Then, for some $t$, $z_t$ has to pass through a branching
point of $\Upsilon_t({\cal I})$, but that is not possible by the
foregoing remark. 

Hence $\Upsilon$ fixes all regular points, but those are dense in
${\cal I}$. Finally, by Theorem~\ref{theo:28up,1}, ${\cal I}$ is dense
in $\tilde{K}_H$.  
\end{proof}

Coming back to the proof of Proposition~\ref{prop:17fa,1}, we observe
that for any $m>n$, $\Upsilon^n(z) = \Upsilon^m(z)$ provided that $z\in
\tilde{K}_{H_{-n}}$. Indeed, both $\Upsilon^m$ and $\Upsilon^n$ conjugate $H_{-n}$ to
$\hat{H}_{-n}$ and so $(\Upsilon^m)^{-1} \circ \Upsilon^n$ provides a
self-conjugacy of $H_{-n}$ and Lemma~\ref{lem:18fa,1} becomes
applicable. 

Now if $\Upsilon = \lim_{k\rightarrow\infty} \Upsilon^{n_{k}}$, then 
$\Upsilon'(z) = \hat\tau \Upsilon(\tau^{-1} z)$ is the limit of the sequence 
$\Upsilon^{n_k+1}$. For $z\in \tilde{K}_{H_{-m}}$ and any $m$, the values of both
sequences at $z$ stabilize. Hence, $\Upsilon(z) = \Upsilon'(z)$ for any $z\in
\bigcup_{m=0}^{\infty} \tilde{K}_{H_{-m}}$ but this set is dense in $\CC$ by
Proposition~\ref{prop:11fa,3} and Lemma~\ref{lem:28up,7}. So,
$\Upsilon=\Upsilon'$ and invariance and the rescaling by $\tau$ 
has been demonstrated.   

Finally, we need to show that $\Upsilon$ conjugates between the
untempered towers as well. Given the invariance under $\tau$, it will
be enough to show that for any $z\in \Omega_- \cup \Omega_+$ we have 
\begin{equation}\label{equ:2xp,1}
 \Upsilon\circ H(z) = {\hat H} \circ \Upsilon(z) 
\end{equation}
for the untempered dynamics. 

Assume first that $z\in \bigcup_{n=0}^{\infty} \tau^n K_H$. Then $z\in
K_{H_{-n}}$ for all $n$ large enough. But $K_{H_{-n}} =
\tilde{K}_{H_{-n}}$ by Lemma~\ref{lem:28up,7} and so $\Upsilon$
conjugates on the entire forward orbit of $z$ by $H^t_{-n}$, which is
the same as the forward orbit by $H_{-n}$. But $H(z)$ is somewhere in
this orbit by  Lemma~\ref{lem:30up,1} and so
relation~(\ref{equ:2xp,1}) holds. 

To finish the proof, invoke Proposition~\ref{prop:11fa,3}.

\subsection{Invariant line-fields}
We will identify {\em measurable line-fields} with differentials in
the form $\nu(z) \frac{d\overline{z}}{dz}$ where $\nu$ is a measurable
function with values on the unit circle or at the origin. 
A line-field is considered
{\em holomorphic} at $z_0$ if for some holomorphic function $\psi$
defined on a neighborhood of $z_0$, we have $\nu(z) =
c \frac{\overline{\psi'(z)}}{\psi'(z)}$ for some constant $c$.  

By a standard reasoning, Proposition~\ref{prop:17fa,1} gives us a
measurable line-field $\mu(z)\frac{d{\overline z}}{dz}$ which
is invariant under the action of $H_n^*$ for any $n$ as well as under
rescaling: $\mu(\tau z) = \mu(z)$.  
  
We will proceed to show that $\mu$ must be trivial, i.e. $0$ almost
everywhere. This will be attained by a typical approach: showing first
that $\mu$ cannot be non-trivial and holomorphic at any $z_0$ for
dynamical reasons, and on the contrary, that it must be  
holomorphic at some point for analytic reasons and because of
expansion. 

\paragraph{Absence of line-fields holomorphic on an open set.}
\begin{lem}\label{linehol}
The line-field $\mu$ cannot be both holomorphic and non-trivial on any open set.
\end{lem}
\begin{proof} Let $\mu$ be holomorphic in a neighborhood $W$.
Since $\mu$ is invariant under $z\mapsto z/\tau$ and since
$\cup_{n\ge 0} \tau^n K_H$ is dense in the plane,
one can assume that $W$ is a neighborhood of
a point $b$ of $K_H$. Moreover, since $b$ is approximated by
preimages of $x_0$ and $\tau x_0$, one can further assume 
that $W$ is a neighborhood of $a$, such that $H^n(a)=x_0$ or
$H^n(a)=\tau x_0$ for some $n\ge 0$, and (shrinking $W$)
that $H^n$ is univalent on $W$. To fix attention, suppose that
$H^n(a)=x_0$, since the other case can be treated in the same way. 
Apply $H^n$ and see that $\mu$
is holomorphic in a neighborhood $W'$ of $x_0$.
Applying $H$ one more time to $W'\cap U$, one sees
that $\mu$ is holomorphic in a neighborhood
of every point of a punctured disk
$D(0,r)\setminus \{0\}$. Now apply the rescalings
$z\mapsto \tau^n z$, $n=0,1,...$. Hence, 
$\mu$ is holomorphic everywhere except for $0$.
In particular, $\mu$ is holomorphic
around $1=H(0)$. Since $H$ is univalent in a neighborhood of $0$,
then $\mu$ is actually holomorphic in the whole disc 
$D(0,r)$. Then $\mu$ cannot be holomorphic around
$H^{-1}(0)=x_0$, a contradiction.
\end{proof}

\paragraph{Construction of holomorphic line-fields.}
Our goal is to prove the following: 
\begin{prop}\label{linetriv}
Suppose that $H$ is an untempered dynamics derived from the ${\cal
  WF}$-class. Assume that $H$ fixes an invariant
line-field $\mu(z) \frac{d\overline{z}}{dz}$, which is additionally
invariant under rescaling: $\mu(\tau z) = \mu(z)$. Then the line-field is
holomorphic at some point. Additionally, it is non-trivial in a
neighborhood of the same point unless $\mu(z)$ vanishes almost everywhere. 
\end{prop}

Construction of holomorphic line-fields is based on the following
analytic idea. This lemma appeared in~\cite{feig} and is repeated here
for the sake of completeness. 

\begin{lem}\label{lem:1gp,1}
Consider  a line-field $\nu_0 \frac{d{\overline
z}}{dz}$ defined on a neighborhood of some point $z_0$ which also is a
Lebesgue (density) 
point for $\nu_0$. Consider a sequence of univalent functions
$\psi_n$ defined on some disk 
$D(z_1,\eta_1)$ chosen so that for every $n$ and a
fixed $\rho<1$ the set $\psi_n(D(z_1,\rho\eta_1))$ covers
$z_0$. In addition, let $\lim_{n\rightarrow\infty} \psi'_n(z_1) = 0$.
Define 
\[ \mu_n(z) \frac{d\overline{z}}{dz}  = \psi_n^*(\nu_0(w))
\frac{d\overline{w}}{dw} \]

Then for some subsequence $n_k$ and a univalent mapping $\psi$ defined
on $D(z_1, \eta_1)$, $\mu_{n_k}(z)$ tend to $\nu_0(z_0)
\frac{\overline{\psi'(z)}}{\psi'(z)}$ on a neighborhood of $z_1$.  
\end{lem}
\begin{proof}
Let us normalize the objects by setting $\hat{\psi}_n :=
|\psi'_n(z_1)|^{-1} \psi_n$ and $\hat{\nu}_n(w) = \nu_0(|\psi'_n(z_1)|
w)$. By bounded distortion, $\hat{\psi}_n(D(z_1,\rho\eta_1))$ contains
some $D(z_0, r_1)$ and is contained in $D(z_0, r_2)$ with $0<r_1 <
r_2$ independent of $n$. By choosing a subsequence, and taking into
account compactness of normalized univalent functions and the fact
that $z_0$ was a Lebesgue point of $\nu_0$, we can assume that
$\hat{\psi}_n$ converge to a univalent function $\psi$ and $\hat{\nu}_n$
converge to a constant line-field $\nu_0(z_0)
\frac{d\overline{w}}{dw}$ almost everywhere. Since 
\[ \mu_n(z) \frac{d\overline{z}}{dz} = \hat{\psi}_n^* (\hat{\nu}_n(w)
\frac{d\overline{w}}{dw} ) \]
for all $n$, we get 
\[ \mu_n(z) \frac{d\overline{z}}{dz} \rightarrow \psi^* (\nu_0(z_0)
\frac{d\overline{w}}{dw}) \]
for $z\in D(z_1, \eta_1\rho)$ which concludes the proof of the Lemma. 
\end{proof}

\paragraph{Proof of Proposition~\ref{linetriv}.}
Start with a Lebesgue point $z_0$ of $\mu$. If the field is
non-trivial, without loss of generality $\mu(z_0) \neq 0$. Also, we
can pick $z_0$ so that it is never mapped on the real line and we can
use Proposition~\ref{prop:11fa,1}. 

We then pick a sequence $z_n$ from
Proposition~\ref{prop:11fa,1} in such a way that the same case happens for all $n$. 
In the first case, we choose a point
$Z$ to be an accumulation point of $\tau^{-m_n} z_{n}$. Without loss
of generality, we suppose that $\tau^{-m_n} z_n \rightarrow Z$. The distance
from $Z$ to $\omega_{\infty}$ is positive and we can denote it by
$2\eta_1$. Then, for any $n$ we can find an inverse branch $\zeta_n$ of the
tower iterate $\chi_n$ mapping $z_0$ to $z_{m_n}$ defined on
$D(z_{m_n},\tau^{m_n}\eta)$. One easily checks that functions
$\psi_n(z) = \zeta_n(\tau^{m_n} z)$ defined on $D(Z,\eta_1)$ satisfy the
hypotheses of Lemma~\ref{lem:1gp,1}. In particular, their derivatives
go to $0$ because $D_{\rho_{\infty}}\chi_n(z_0)$ go to $\infty$ by
Lemma~\ref{lem:5gp,2}.   

To consider the second case of Proposition~\ref{prop:11fa,1}, fix
attention on some $n$. The first observation is that without loss of
generality $|H_{-m_n}(z_n)| < R' |\tau|^{m_n}$ with some $R' < R$
independent of $n$. Indeed, all points on the circle $C(0,
|\tau|^{m_n} R)$ are in distance $\eta |\tau|^{m_n}$ from
$\omega_{\infty}$ for some $\eta$ positive. So if this additional
property fails for infinitely many $n$, we can reduce the situation to
the first case already considered. 

Now the key observation is that for every $n$ the point $z_n$ has a
simply connected neighborhood $Y_n$, a point $y_n \in Y_n$ such that
the distance in the hyperbolic metric of $Y_n$ from $z_n$ to $y_n$ is
bounded independently of $n$. Finally, $Y_n$ is  mapped univalently 
by $H_{-m_n}$ so that for some integer 
$p_n$ and $\eta>0$ which is independent of
$n$ its image covers $\tau^{p_n}(D(i,\eta))$ with $H_{-m_n}(y_n) = \tau^{p_n}
i$. To choose $Y_m$ invoke Lemma~\ref{lem:8xa,1} after rescaling. In Lemma~\ref{lem:8xa,1}, set 
$M = \log |\tau|$ which will guarantee that $W_M$ contains $\tau^{-p_n} i$. Then Lemma~\ref{lem:8xa,1}
provides a uniform hyperbolic bound. Univalence can be obtained by restricting to the appropriately smaller neighborhood of the 
path joining $z_{m_n}$ to $\tau^{p_n} i$. 

Once $y_n, Y_n, p_n$ were chosen, we easily conclude the proof. Let
$R_n: D(0,1)\to Y_n$ 
be Riemann maps of regions $Y_n$ with $R_n(0) = y_n$. Then we can
set $\psi_n = (\chi_n)^{-1} \circ R_n$ where $\chi_n$ are maps
specified in Proposition~\ref{prop:11fa,1}. Maps $\psi_n$ satisfy the
conditions of Lemma~\ref{lem:1gp,1}. In particular, 
$|R^{-1}_n(z_n)|$ is bounded independently of $n$ as a consequence of the
construction of $Y_n$.

From this and Proposition~\ref{prop:11fa,1},
the derivatives of $\psi_n$ at $R_n^{-1}(z_n)$ go to $0$, and then the
same can be said of $\psi'_n(0)$ by bounded distortion. So, by passing
to a subsequence, we get that $R_n^* (\mu(z)
\frac{d\overline{z}}{dz})$ tend a.e. to a holomorphic line-field $\nu
\frac{d\overline{w}}{dw}$ on a neighborhood of $0$. 

To finish the proof, we ignore the fact that a subsequence has been
chosen and consider mappings $T_n := \tau^{-p_n} H_{-m_n} \circ R_n$
defined on the unit disk. We have $T_n^* (\mu(z)
\frac{d\overline{z}}{dz}) = R_n^* (\mu(z) \frac{d\overline{z}}{dz})$
for every $n$. Maps $T_n$ are all univalent and have been normalized
so that $T_n(0) = i$ and the image of $D(0,1)$ under $T_n$ contains
$D(i, \eta)$ for a fixed $\eta>0$, but avoids $0$. Then $T_n$ is a
compact family of univalent maps and has a univalent limit $T$. Then
it develops that $\mu$ in a neighborhood of $i$ is the image under $T$
of the holomorphic line-field $\nu$ from a neighborhood of $0$, hence
is holomorphic.   

\paragraph{Proof of Theorem~\ref{theo:13xp,3}.}
Suppose that class ${\cal EWF}$ contains two mappings, $\phi$ and $\hat{\phi}$. We construct their 
towers $(H_{-n})$ and $(\hat{H}_{-n})$. By Proposition~\ref{prop:17fa,1}, the towers are quasiconformally conjugated which gives rise 
a line-field invariant under $H^*_{-n}$. If this line-field is non-trivial, then by Proposition~\ref{linetriv} it is holomorphic on some open set, but then
one reaches a contradiction with Lemma~\ref{linehol}. If the line-field is trivial, then the quasiconformal conjugacy is affine. But it fixes $0$ and $1$, hence 
$\phi=\tilde{\phi}$. 
  
\paragraph{Proof of Theorem~\ref{theo:12xp,1}.}
By Facts~\ref{fa:13xp,1} and~\ref{fa:13xp,2} we can pick $k_{\ell}$ so that for every $n_{\ell}\geq k_{\ell}$ and every $\ell$, the $C^0$ distance between 
${\cal R}^{n_{\ell}}(\psi_{\ell})$ and the fixed point $H_{\ell}$ tends to $0$ as $\ell\to\infty$. 
Then, by Theorems~\ref{theo:13xp,1} and~\ref{theo:13xp,2}, every subsequence of ${\cal R}^{n_{\ell}}(\psi_{\ell})$ 
has a subsequence convergent to a ${\cal EWF}$ map. But by Theorem~\ref{theo:13xp,3}, it means that the whole sequence converges to the unique member of the ${\cal EWF}$-class.

\section{Proof of Theorem~\ref{theo:13xp,4}}

The idea is to introduce the presentation function $\Pi_\ell$ 
for finite $\ell$
and to present the post-critical set of $H_\ell$
as a ``non-escaping set'' of the presentation function on some proper 
subinterval
of its domain of definition. For the limit map and for its presentation
function $\Pi$, it is done in Lemma~\ref{lem:28ua,4}.
Then we pass to the limit as $\ell\to \infty$ 
using Theorem~\ref{theo:12xp,1}. It is important to keep in mind that $X_{\ell} \to \frac{x_0}{\tau^2}$
when this limit is taken. Thus, points $\tau^2_{\ell} X_{\ell}$ and $x_{\ell}$ merge in the limit to a parabolic fixed point
of $G$. 

Fix a finite $\ell$. 
Recall that the point $X_\ell$ is defined as the unique solution of the
equation $\phi_\ell(X_\ell)=\tau_\ell X_\ell$, see Lemma~\ref{covcontin}. 
Recall also that $x_\ell$ denotes the critical point of $\phi_\ell$.
Note that $\tau_\ell^2 X_\ell<x_\ell<X_\ell$.
The map $H_\ell$ consists of a pair of maps.
It is defined  as the map $\phi_\ell$ on
$(\tau_\ell^2 X_\ell, \tau_\ell X_\ell]$ and 
$\tau_\ell\circ \phi_{\ell}\circ \tau_\ell^{-1}$ on 
$(\tau_\ell X_\ell, \tau_\ell^3 X_\ell)$.

Define the presentation function $\Pi_\ell$ similar to $\Pi$
as follows. $\Pi_\ell$ is defined  
on the interval $(\tau_\ell^2 X_\ell,\tau_\ell^3 X_\ell)$. 
For $\tau_\ell^2 X_\ell<x\leq \tau_\ell X_\ell$, 
$\Pi_\ell(x) = \tau_\ell x$. For $\tau_\ell X_\ell<x<\tau_\ell^3 X_\ell$, 
$\Pi_\ell(x) = 
\tau_\ell\circ \phi_{\ell}\circ \tau_\ell^{-1}(x)$. 

Like in the limit case, 
it is post-critically finite since $\Pi_\ell(\tau_\ell x_\ell) = 0$ 
and $0$ is a repelling fixed point of $\Pi_\ell$.
Also, the points
$1,\tau_\ell^{-1}$ form a periodic orbit under $\Pi_\ell$. 

Denote by $\omega_\ell$ the post-critical set of $H_\ell$.
It is the closure of the forward iterates $H^i_\ell(0)$, $i\ge 0$,
of the critical value $0$.

As it follows directly from the definition of Fibonacci combinatorics
for covering maps, we have the following fact:
\begin{equation}\label{w}
\omega_\ell\subset [x_\ell, \tau_\ell x_\ell].
\end{equation}
Note that $[x_\ell, \tau_\ell x_\ell]$ is a proper subset 
of the domain of definition of $H_\ell$. 
Using the fact~(\ref{w}) and repeating word by word the proof of 
Lemma~\ref{lem:8xp,1} we get the following.
\begin{lem}\label{wpre}
A point $x\in (x_\ell,\tau x_\ell)$ is equal to 
$H_\ell^j(0)$, $j>0$, if and only
if $\Pi_\ell^k(x)=1$ for some $k\geq 0$. 
\end{lem}

Similar to the limit case, we consider the first return map
of $\Pi_\ell$ to the open interval $(x_\ell, X_\ell)$.
Denote this first return map by $R_\ell$. Since $\Pi^2_\ell$
maps $(x_\ell, X_\ell)$ onto $(\tau_\ell^2 X_\ell, 0)$ and
$\Pi^2_\ell=\tau_\ell\circ \phi_\ell$ on $(x_\ell, X_\ell)$, the following
statement follows easily.
\begin{lem}\label{fret}
Denote $\psi_n=\tau_\ell^{2n+1}\phi_\ell$, $n=0,1,...$.
For every $n\ge 0$,
there exists an interval $K^n_\ell\subset (x_\ell, X_\ell)$ such that
$\psi_n(K_\ell^n)=(x_\ell, X_\ell)$, $n=0,1,...$. The intervals $K_\ell^n$
are pairwise disjoint. Domain of definition of the map $R_\ell$ is the union
of $K_\ell^n$, $n\ge 0$, so that $R_\ell$ on $K_\ell^n$ is equal to  $\psi_n$.
Each branch of $R_\ell$ has negative Schwarzian derivative and
extends diffeomorphically to a
fixed interval which contains $[x_\ell,X_\ell]$ in its interior.
\end{lem}

The next statement is crucial.
Let $\tilde\omega_\ell=\omega_\ell\cap [x_\ell, X_\ell]$. 
Introduce also the non-escaping set
of $R_\ell$:
$$J(R_\ell)=\{x\in (x_\ell, X_\ell): R_\ell^n(x)\in (x_\ell, X_\ell),
n\ge 0 \}.$$
\begin{lem}\label{wfret}
(a) A point $x\in (x_\ell,X_\ell)$ is equal to 
$H_\ell^j(0)$, $j>0$, if and only
if $R_\ell^k(x)=1/\tau_\ell$ for some $k\geq 0$. 
(b) $\tilde\omega_\ell$ coincides with the closure of 
the non-escaping set $J(R_\ell)$.
\end{lem}
\begin{proof}
(a) follows directly from Lemma~\ref{wpre} and from the fact that
$R_\ell$ is the first return map. In turn, (b) follows from (a)
and from the fact that the preimages of any point in $J(R_\ell)$
by all $R^n_\ell$, $n\ge 0$, are dense in $J(R_\ell)$. Then
we use that
$1/\tau_\ell\in J(R_\ell)$ because $\Pi^2_\ell(1/\tau_\ell)=1/\tau_\ell$.
\end{proof}

As $\ell\to \infty$, then $x_\ell\to x_0$
and $X_\ell\to x_0/\tau^2$.
Moreover, given $n\ge 0$, the interval $K_\ell^n$ tends $K^n$,
so that the intervals $K^n$, $n\ge 0$, form a partition of
$(x_0, x_0/\tau^2)$. 
We can elaborate Lemma~\ref{lem:8xp,2} as follows.
The first return map $R$ of the limit presentation function
$\Pi$ into $(x_0,x_0/\tau^2)$ is defined on $\cup_{n\ge 0}K^n$,
and $R|_{K^n}=\tau^{2n+1}\phi$.
Clearly, the closure of the non-escaping set $J(R)$ of $R$
is the interval $[x_0, x_0/\tau^2]$, which is just the
intersection $\tilde\omega$ of the post-critical set 
$\omega=[x_0, \tau x_0]$ of $H$ with the range of $R$.
In particular, Hausdorff dimension $HD(\tilde\omega)$
of $\tilde\omega$ is $1$.
On the other hand, since $R$ is an infinite conformal
iterated function system, see~\cite{MU},
$$1=HD(\tilde\omega)=\sup\{HD(J(R(A))\},$$
where the supremum is taken over all finite subsets $A$ of 
non-negative integers, $R(A)$ denotes the finite iterated function system
formed by the subset of the branches of $R$ with the indexes in the set $A$,
and $J(R(A))$ is the closure of non-escaping set of $R(A)$.
Given $\delta>0$, we find $A$ such that $HD(J(R(A))>1-\delta$.
On the other hand, by Theorem~\ref{theo:12xp,1}, any branch
$\tau_\ell^{2n+1}\phi_\ell$ of $R_\ell$ tends to the corresponding
branch  $\tau^{2n+1}\phi$ of $R$. Therefore, since $A$ is a finite set
of indexes $n$, 
$HD(J(R_\ell(A))\to HD(J(R(A))$ as $\ell\to \infty$, with the obvious notation
for $R_\ell(A)$ as a finite subset of $R_\ell$ with the indexes in the same
finite set $A$. Thus, for any $\ell$ large enough,
$$HD(\omega_\ell)=HD(\tilde\omega_\ell)=HD(J(R_\ell))\ge 
HD(J(R_\ell(A))>1-\delta.$$
Since $\delta>0$ is arbitrary, the theorem is proved.

\paragraph{Acknowledgments.}
We thank the referee for insightful comments and Adam \'{S}wi\c{a}tek for making the illustrations.


\begin{thebibliography}{99}
\bibitem{carleson} Carleson, L. \& Gamelin, T.: {\em Complex
dynamics}, Springer-Verlag, New York (1993) 
\bibitem{bahl} Beurling, A\& Ahlfors, L.: {\em The boundary correspondence under quasiconformal mappings}, Acta Math. {\bf 96}
  (1956) , 125-142
\bibitem{expcirc} Dixon, T.W., Kelly, B.G., Briggs, K.M.: {\em
On the universality of singular circle maps}, Physics Letters A, {\bf 231},
359-366 (1997)
\bibitem{defaria} De Faria, E: {\em Asymptotic rigidity of scaling ratios
for critical circle mappings}, Ergod. Th. Dynam. Sys. {\bf 19}, 995-1035
(1999)
\bibitem{demelofaria} De Faria, E. \& De Melo, W.:
{\em Rigidity of critical circle mappings II}, Journal of the AMS
{\bf 13}, 343-370 (1999)
\bibitem{destrien} De Melo, W. \& Van Strien, S.: {\em One-dimensional
dynamics}, Ergebnisse Series {\bf 25}, Springer-Verlag (1993) 
\bibitem{DH} Douady, A. \& Hubbard, J.H.: {\em On the dynamics of
polynomial-like mappings}, Ann. Sci. \'{E}cole Norm. Sup. (Paris) {\bf
18} (1985), 287-343
\bibitem{EW} Eckmann, J.-P. \& Wittwer, P.:
{\em Computer Methods and Borel Summability Applied
to Feigenbaum's equation}, Lecture Notes in Physics 227,
Springer-Verlag, 1985
\bibitem{feig0} Feigenbaum, M.: {\em Qualitative universality
for a class of non-linear transformations}, J. Stat. Phys.
{\bf 19} (1978), 25-52
\bibitem{feig1} Feigenbaum, M.: {\em The universal metric properties
of non-linear transformations}, J. Stat. Phys.
{\bf 21} (1979), 669-706
\bibitem{leprzy} Levin, G. \& Przytycki, F.: {\em On Hausdorff dimension of some
Cantor attractors}, Israel Jour. of Math. {\bf 149} (2005), 185-198
\bibitem{feig} Levin, G. \& \'{S}wi\c{a}tek, G.: {\em
Dynamics and universality of unimodal mappings with infinite criticality},
Comm. Math. Phys. {\bf 258}, 103-133 (2005)
\bibitem{LSw2} Levin, G. \& \'{S}wi\c{a}tek, G.
: {\em Hausdorff dimension of Julia sets
of Feigenbaum polynomials with high
criticality}, Comm. Math. Phys. 258, 135-148 (2005)
\bibitem{lsuniv} Levin, G. \& \'{S}wi\c{a}tek, G.:
{\em Universality of critical circle covers}, 
Comm. Math. Phys. {\bf 228}, 371-399 (2002)
\bibitem{MU} Mauldin, D., Urbanski, M.: {\em Dimensions and measures
in infinite iterated function systems}, Proc. London Math. Soc. 
{\bf 73}, 105-154 (1996)
\bibitem{mcm} Mc Mullen, C.: {\em Complex dynamics and
renormalization}, Ann. of Math. Studies {\bf 135}, Princeton
University Press (1994)
\bibitem{profesorus} Mc Mullen, C.: {\em Renormalization and
3-manifolds which fiber over the circle}, Ann. of Math. Studies {\bf
142}, Princeton University Press (1998)

\bibitem{Su} Sullivan, D.: {\em Bounds, quadratic differentials and
renormalization conjectures}, in: {\em Mathematics into the
Twenty-First Century}, AMS Centennial Publications (1991) 
\bibitem{sulcio} Sullivan, D. {\em Quasiconformal homeomorphisms and
dynamics I: a solution of Fatou-Julia problem on wandering domains},
Ann. Math. {\bf 122} (1985), 401-418
\bibitem{gregbra} \'{S}wi\c{a}tek, G.: 
{\em On critical circle homeomorphisms}, Bol. Bras. Math. Soc., {\bf 29},
329-351 (1998)
 
\end{thebibliography}
\end{document}